\documentclass[reqno]{amsart}
\usepackage{bm,amsmath,amsthm,amssymb,mathtools,verbatim,amsfonts,tikz-cd,mathrsfs}
\usepackage{enumitem}
\usepackage{float}

\usepackage[hidelinks,colorlinks=true,linkcolor=blue, citecolor=black,linktocpage=true]{hyperref}
\usepackage{graphicx}
\usepackage[alphabetic]{amsrefs}
\usepackage{caption}
\usepackage{morefloats}
\usepackage{chngcntr}

\counterwithin{figure}{section}

\newtheorem{thm}{Theorem}[section]
\newtheorem{prop}[thm]{Proposition}

\newtheorem{cor}[thm]{Corollary}
\newtheorem{lem}[thm]{Lemma}

\theoremstyle{definition}
\newtheorem{define}[thm]{Definition}

\theoremstyle{remark}
\newtheorem{rem}[thm]{Remark}

\newtheorem{question}[thm]{Question}

\newcommand{\ve}[1]{\boldsymbol{\mathbf{#1}}}
\newcommand{\R}{\mathbb{R}}

\newcommand{\ul}[1]{\underline{#1}}

\newcommand{\Z}{\mathbb{Z}}
\newcommand{\N}{\mathbb{N}}

\renewcommand{\d}{\partial}
\renewcommand{\subset}{\subseteq}

\renewcommand{\tilde}{\widetilde}
\renewcommand{\bar}{\overline}

\newcommand{\iso}{\cong}

\DeclareMathOperator{\gr}{{gr}}

\DeclareMathOperator{\id}{{id}}

\DeclareMathOperator{\Int}{{int}}

\DeclareMathOperator{\im}{{im}}

\DeclareMathOperator{\Spin}{{Spin}}

\DeclareMathOperator{\Sym}{{Sym}}

\DeclareMathOperator{\Tors}{{Tors}}

\DeclareMathOperator{\Cone}{Cone}

\newcommand{\bF}{\mathbb{F}}

\newcommand{\bH}{\mathbb{H}}

\newcommand{\bK}{\mathbb{K}}
\newcommand{\bL}{\mathbb{L}}

\newcommand{\bT}{\mathbb{T}}

\newcommand{\CP}{\mathbb{CP}}

\newcommand{\cA}{\mathcal{A}}

\newcommand{\cC}{\mathcal{C}}
\newcommand{\cD}{\mathcal{D}}

\newcommand{\cF}{\mathcal{F}}
\newcommand{\cG}{\mathcal{G}}
\newcommand{\cH}{\mathcal{H}}

\newcommand{\cM}{\mathcal{M}}

\newcommand{\cR}{\mathcal{R}}
\newcommand{\cS}{\mathcal{S}}
\newcommand{\cT}{\mathcal{T}}

\newcommand{\frs}{\mathfrak{s}}

\newcommand{\scA}{\mathscr{A}}
\newcommand{\scC}{\mathscr{C}}
\newcommand{\scI}{\mathscr{I}}
\newcommand{\scU}{\mathscr{U}}
\newcommand{\scV}{\mathscr{V}}

\newcommand{\cCFL}{\mathcal{C\! F\! L}}
\newcommand{\cCFK}{\mathcal{C\! F\! K}}

\newcommand{\cHFK}{\mathcal{H\! F\! K}}

\newcommand{\CFK}{\mathit{CFK}}

\newcommand{\CFI}{\mathit{CFI}}
\newcommand{\HFI}{\mathit{HFI}}

\newcommand{\PD}{\mathit{PD}}

\newcommand{\xs}{\ve{x}}
\newcommand{\ys}{\ve{y}}
\newcommand{\zs}{\ve{z}}
\newcommand{\ws}{\ve{w}}

\newcommand{\as}{\ve{\alpha}}
\newcommand{\bs}{\ve{\beta}}

\renewcommand{\a}{\alpha}
\renewcommand{\b}{\beta}
\newcommand{\g}{\gamma}

\usepackage{leftidx}

\title{New  Heegaard Floer slice genus and clasp number bounds}

\author{Andr\'{a}s Juh\'{a}sz}
\address{Mathematical Institute, University of Oxford, Andrew Wiles Building,
Radcliffe Observatory Quarter, Woodstock Road, Oxford, OX2 6GG, UK}
\email{juhasza@maths.ox.ac.uk}

\author{Ian Zemke}
\address{Department of Mathematics\\Princeton University\\  Princeton, NJ 08544, USA}
\email{izemke@math.princeton.edu}

\begin{document}

%\subjclass[2010]{57R58; 57R55; 57M27}
%\keywords{4-manifolds, exotic surfaces, Heegaard Floer homology}

\begin{abstract}
We define several concordance invariants using knot Floer homology which give improvements over known slice genus and clasp number bounds from Heegaard Floer homology. We also prove that the involutive correction terms of Hendricks and Manolescu give both a slice genus and a clasp number bound.
\end{abstract}

\maketitle

%\tableofcontents

\section{Introduction}

If $K\subset S^3$ is a knot, the \emph{4-dimensional clasp number} $c_4(K)$, introduced by Shibuya~\cite{Shibuya}, 
is the smallest integer $c$ such that there is a normally immersed disk in $B^4$ bounding $K$ with $c$ double points.
(A surface in a 4-manifold $W$ is \emph{normally immersed} if it is the image of an immersion $f \colon S \to W$ such that $f(S) \cap \d W = f(\d S)$, $f$ is transverse to $\d W$, and all self-intersections of $f$ are transverse double points in $\Int(W)$.)
There are refinements $c_4^+(K)$ (respectively $c_4^-(K)$), which are the minimal number of positive (respectively negative)  double points which appear in any such immersed disk. Of course,
\[
c_4^+(K)+c_4^-(K)\le c_4(K).
\]

Reversing the orientation of a normally immersed surface in $B^4$ does not change the sign of a double point, and hence changing the orientation of a knot does not affect $c_4^+$ or $c_4^-$. However, if $-K$ denotes the mirror of $K$, then
\[
c_4^+(K)=c_4^-(-K).
\]
Another basic inequality is
\[
g_4(K)\le c_4(K)\le u(K),
\]
where $g_4(K)$ denotes the slice genus, and $u(K)$ the unknotting number.

%More generally, we will study the set of pairs $(c,g)$ such that there is an oriented, normally immersed surface in $B^4$ that bounds a given knot $K$, and has genus $g$ and $c$ double points. \marginpar{AJ: actually, this has been removed IZ: We don't really say anything to this effect beyond what could say using the clasp number.}

The first example of a knot $K$ with $g_4(K) < c_4(K)$ was given by Murakami and Yasuhara~\cite{MY}.
They used the linking form of the branched double cover to show that $g_4(8_{16}) = 1$ and $c_4(8_{16}) = 2$. 
Using the Heegaard Floer $d$-invariants of branched double covers, Owens and Strle \cite{OwensStrle} gave lower bounds on the clasp number, and completed the computation of the clasp number for all prime knots with ten or fewer crossings.
Livingston~\cite{LivingstonTL} used the more classical Tristram--Levine signatures to give lower bounds on the clasp
number. The lower bounds arising this way are never greater than $2g_4(K)$, which motivated him to ask the following question,
which is still open:

\begin{question}\label{qn:c-g}
Are there knots with $c_4(K) > 2g_4(K)$?
\end{question}

More recently, Kronheimer and Mrowka \cite{KH-Instantons-concordance} constructed  a 1-parameter family of concordance invariants using instanton knot Floer homology, from which both slice genus and clasp number bounds may be obtained. They asked if
this could be used to show that $c_4(K) > g_4(K)$ where other methods have failed. For example, $g_4(7_4) = 1$ and $c_4(7_4) = 2$,
but previously the only thing known about $c_4(\#^n 7_4)$ was that it lies in the interval $[n, 2n]$. 
Daemi and Scaduto \cite{DaemiScadutoClasp} have recently shown that
\[
c_4(\#^n 7_4) - g_4(\#^n 7_4) \ge n/5
\]
using instantons.

The motivation of our present work is to investigate these questions from the perspective of Heegaard Floer homology.

\subsection{The concordance invariants $Y_n$}

In this paper, we introduce a family of concordance invariants $Y_{n}(K)\in \N$, indexed by $n\in \N$. 
We will prove the following:

\begin{thm}\label{thm:Y-genus-clasp}
 Suppose $K$ is a knot in $S^3$.
\begin{enumerate}
\item If $0\le n\le g_4(K)$, then
\[
Y_n(K)\le \left \lceil \frac{g_4(K)-n}{2} \right\rceil,
\]
and $Y_n(K) = 0$ if $n > g_4(K)$. 
\item If $0\le n\le c_4^+(K)$, then
\[
Y_n(K)\le \left \lceil \frac{c_4^+(K)-n}{2}\right \rceil,
\]
and $Y_n(K) = 0$ if $n > c_4^+(K)$.
\end{enumerate}
If $n\ge g_4(K)$ or $n\ge c^+(K)$, then $Y_n(K)=0$.
\end{thm}

The invariants $Y_{n}$ are formally similar to Rasmussen's correction terms of large surgeries \cite{RasmussenKnots}, which we denote by $V_n(K)$. In comparison with Theorem~\ref{thm:Y-genus-clasp},
\begin{equation}
V_n(K)\le \left \lceil \frac{g_4(K)-n}{2} \right\rceil. \label{eq:rasmussen-bound}
\end{equation}
if $0 \le n \le g_4(K)$, and $V_n(K) = 0$ for $n > g_4(K)$;
see \cite{RasmussenGodaTeragaito}*{Theorem~2.3}. We will prove that
\begin{equation}
Y_n(K)\ge V_n(K),
\label{eq:comparison-with-Vk}
\end{equation}
and also give examples where the invariants $Y_n(K)$ give better genus bounds than the invariants $V_n(K)$.
While $Y_n(K)$ is a concordance invariant,
it is not a concordance homomorphism, as $Y_n(-K) \neq -Y_n(K)$, unless they both vanish.
Furthermore, often a knot $K$ has the property that both
 $Y_n(K)$ and  $Y_m(-K)$ are non-zero for some $n,m\in \N$. 
This allows one to use $Y_n$ to give clasp number bounds better than $g_4$,
but at most $2g_4$; cf.~Question~\ref{qn:c-g}.

Rasmussen's genus bound in~\eqref{eq:rasmussen-bound} was originally proven by considering the maps induced by surgeries. 
An alternate proof \cite{ZemAbsoluteGradings}*{Section~10} may be given by considering the link cobordism maps due to the first author~\cite{JCob} and the second author~\cite{ZemCFLTQFT}, associated to the surface itself, for different choices of decorations. In this paper, we prove a strong result relating the elements induced by different decorations of the surface. See Section~\ref{sec:bounds-and-Y}. A consequence of our techniques is that if there is a genus $n$ slice surface for $S$, then there is a grading preserving chain map from the staircase complex $\cS_n$ to $\cCFK^-(K)$, which becomes an isomorphism on homology after inverting $U$. Here, $\cS_n$ denotes the staircase complex with $n$ steps, which are all of length 1.

We recall that Hom and Wu \cite{HomWuNu+} define an invariant $\nu^+(K)$, which is the minimal integer $s$ such that $V_s(K)=0$. The invariant $\nu^+(K)$ is a lower bound for the slice genus.  We define the following analog of $\nu^+$:
\[
\omega^+(K):=\min \left\{ n\in \N: Y_{n}(K)=0\right\}.
\]
Theorem~\ref{thm:Y-genus-clasp} implies that
\[
\omega^+(K)\le \min\{g_4(K), c_4^+(K)\}.
\]
Furthermore, $\omega^+(-K)\le c_4^-(K)$, where $-K$ denotes the mirror of $K$. Summarizing, we have the following:
\begin{thm}\label{thm:omega-bound-c}
 If $K$ is a knot in $S^3$, then
 \[
\omega^+(K)+\omega^+(-K)\le c_4(K).
 \]
\end{thm}

Theorem~\ref{thm:omega-bound-c} should be compared to a result of Bodn\'{a}r, Celoria, and Golla \cite{BCGCobordisms}*{Proposition~1.5}, which states that
\begin{equation}
\nu^+(K)+\nu^+(-K)\le c_4(K) \label{eq:BCG-bound}
\end{equation}
for any knot $K\subset S^3$.
In Section~\ref{sec:examples}, we give some example computations where $\omega^+(K)$ is larger than $\nu^+(K)$. 

%Our proof makes use of the link cobordism maps described by the first author~\cite{JCob} and the second author~\cite{ZemCFLTQFT}. 
%The most basic strategy one can employ to understand double points is to \emph{smooth} a double point by increasing the genus. A slightly more refined approach is to smooth only the positive double points, and blow up at the negative double points. This approach will lead to the bound shown in ~\eqref{eq:BCG-bound}. Our approach is a refinement, which is to consider the elements induced by $g$ different dividing sets on the slice surface. This gives
%
%
% Our approach is a refinement of this strategy, which builds off the observation that the cobordism maps for a surface obtained by smoothing have a particularly weak dependence on the dividing set. We encode this algebraically as the existence of an $\bF[U]$-non-torsion cycle in a chain complex we call $\cS_{c,g}(K)$, whose algebraic $d$-invariant determines $Y_{c,g}(K)$. Since this is more restrictive algebraically, $\omega^+(K)$ gives a potentially better bound than $\nu^+(K)$.

We also prove the following bound on the 4-dimensional clasp number:

\begin{thm} \label{thm:clasp-upsilon}
Let $K$ be a knot in $S^3$. Then
\begin{equation}
c_4(K) \ge \left(\max_{t\in [0,1]} \Upsilon_K(t)/t\right) -\left(\min_{t\in [0,1]} \Upsilon_K(t)/t\right). \label{eq:Upsilon-bound-intro}
\end{equation}
\end{thm}

Our proof of Theorem~\ref{thm:clasp-upsilon} uses  only the bound in~\eqref{eq:BCG-bound}, as well a general relation between $\nu^+(K)$ and the invariant $\Upsilon_K(t)$  of Ozsv\'{a}th, Stipsicz, and Szab\'{o} \cite{OSSUpsilon}.
As a topological application, we prove the following:

\begin{cor}\label{cor:unbounded-difference}
 The knot $K=T_{2,11}\#-T_{4,5}$ satisfies $g_4(\#^n K)= n$ and $c_4(\#^n K)\ge 2n$. In particular
 \[
\lim_{n\to \infty} \left( c_4(\#^n K)-g_4(\#^n K)\right)=\infty.
 \]
\end{cor}

The genus computation in Corollary~\ref{cor:unbounded-difference} is due to Livingston and Van Cott \cite{LivingstonVanCott}.

Feller and Park \cite{FellerParkClasp} have informed us
that Corollary~\ref{cor:unbounded-difference} can also be obtained using Levine--Tristram
signatures, using the formula
\[
\max_{t\in [0,1]} \sigma_K(t) - \min_{t\in [0,1]} \sigma_K(t) \le 2c_4(K),
\]
where $\sigma_K$ is the signature function, and the maximum and minimum are taken in the complement of the set of jumps of the signature function.

In Section~\ref{sec:examples}, we give several examples where the invariants $Y_n(K)$ give better genus and clasp bounds than both the invariants $V_n$, and the signature function.

\subsection{An invariant from the $\scU\scV=0$ version of knot Floer homology}

We also obtain an invariant on the $\scU\scV=0$ version of knot Floer homology. We call our invariant $\omega (K)$. The invariant $\omega (K)$ takes values in $\{\tau(K),\tau(K)+1\}$, and is reminiscient of the invariant $\nu (K)$ defined by Ozsv\'{a}th and Szab\'{o} \cite{OSRationalSurgeries}*{Definition~9.1}, which also takes values in $\{\tau(K),\tau(K)+1\}$.

\begin{thm}
The invariants $\omega (K)$ and $\nu (K)$ satisfy
\[
\nu (K)\le \omega (K)\le g_4(K).
\]
\end{thm}

In Section~\ref{sec:examples}, we will give an example where $\nu (K)+1=\omega (K)$.
The advantage of using $\omega $ over $\omega^+$ is that the former can be computed
algorithmically from the $\scU\scV=0$ complex determined by the software of Ozsv\'ath and Szab\'o \cite{SzaboProgram}, which is based on the bordered knot invariant \cite{OSBorderedKauffmanStates}.

\subsection{Involutive correction terms}

We additionally consider the involutive correction terms $\bar{V}_0(K)$ and $\underline{V}_0(K)$
of large surgery due to Hendricks and Manolescu \cite{HMInvolutive}, which satisfy
\[
\bar{V}_0(K)\le V_0(K)\le \ul{V}_0(K).
\]
 The involutive correction terms of large surgery are well known not to give a lower bound on the slice genus in the same manner as $V_0$ in equation~\eqref{eq:rasmussen-bound}. For example, $T_{2,3}\# T_{2,3}$ has slice genus 2, but $\underline{V}_0=2$; see \cite{HMInvolutive}*{Theorem~1.7}. In contrast, our techniques imply the following:

\begin{thm}\label{thm:involutive} Suppose $K$ is a knot in $S^3$.
\begin{enumerate}
\item Then
\[
-\left\lceil \frac{g_4(K)+1}{2}\right\rceil \le \bar{V}_0(K)\le \ul{V}_0(K)\le \left \lceil \frac{g_4(K)+1}{2} \right \rceil,
\]
\item and
\[
-\left\lceil \frac{c_4^-(K)+1}{2}\right\rceil \le \bar{V}_0(K)\le \ul{V}_0(K)\le \left \lceil \frac{c_4^+(K)+1}{2} \right \rceil.
\]
\end{enumerate}
\end{thm}

The knot considered in Section~\ref{sec:computation} is an example where the bounds from $\underline{V}_0(K)$ and $\overline{V}_0(K)$ are stronger than those from equation~\eqref{eq:rasmussen-bound}. See Remark~\ref{rem:more-on-example-knot}.

\subsection{Acknowledgments}

The authors would like to thank Maciej Borodzik, Kristen Hendricks, Jennifer Hom, and Beibei Liu for helpful conversations. We thank Peter Feller, Marco Golla, Maggie Miller, and JungHwan Park for helpful comments on an earlier version of our paper. We also thank Ciprian Manolescu, who wrote (long ago) portions of the code for computing the involutive invariants.

This project has received funding from the European Research Council (ERC)
under the European Union's Horizon 2020 research and innovation programme
(grant agreement No 674978).
The second author was supported by an NSF Postdoctoral Research Fellowship (DMS-1703685).

\section{$\Upsilon_K(t)$ and the clasp number}
\label{sec:example-1}
In this section, we give a family of knots where $c_4-g_4$ increases without bound. Our proof is an application of the following general result:

\begin{thm}\label{thm:Upsilon-bound} Suppose $K$ is a knot in $S^3$. Then
\[
c_4(K)\ge \left(\max_{t\in [0,1]} \Upsilon_K(t)/t\right) -\left(\min_{t\in [0,1]} \Upsilon_K(t)/t\right).
\]
\end{thm}

\begin{proof} According to \cite{OSSUpsilon}*{Proposition~4.7}, we have
\[
-t \nu^+(K)\le \Upsilon_K(t)
\]
for $t\in [0,1]$; cf.~\cite{OSSUpsilon}*{Proposition~2.13}. Hence
\begin{equation}
\nu^+(K)\ge \max_{t\in [0,1]} -\Upsilon_K(t)/t=-\min_{t\in [0,1]} \Upsilon_K(t)/t.
\end{equation}
Mirroring $K$, we obtain similarly
\begin{equation}
\nu^+(-K)\ge \max_{t\in [0,1]} \Upsilon_K(t)/t.
\end{equation}
Using Bodn\'{a}r, Celoria, and Golla's bound in ~\eqref{eq:BCG-bound}, we obtain
\[
c_4(K)\ge \nu^+(K)+\nu^+(-K)\ge \max_{t\in [0,1]} \Upsilon_K(t)/t-\min_{t\in [0,1]} \Upsilon_K(t)/t,
\]
completing the proof.
\end{proof}

As a corollary, we prove the following, which is stated as Corollary ~\ref{cor:unbounded-difference} in the introduction:
\begin{cor} Let $K_r=T_{2,10r+1}\#-T_{4,4r+1}$ for $r\ge 1$.
Then
\[
 g_4(\#^n K_r)=nr\quad \text{and}\quad c_4(\#^n K_r)\ge 2nr.
\]
\end{cor}
\begin{proof}
 Livingston and Van Cott \cite{LivingstonVanCott}*{Theorem~17} proved that the knot $K_r$ has slice genus $r$, and hence 
 \begin{equation}
 g_4(\#^n K_r)\le nr. \label{eq:genus-bound}
 \end{equation}
On the other hand, they compute using \cite{OSSUpsilon}*{Theorem~6.2} that
 \begin{equation}
\tau(K_r)=-r\quad \text{and} \quad \Upsilon_{K_r}(1)=-r. \label{eq:tau-upsilon-Kr}
 \end{equation}
In particular $\tau(\#^n K_r)=-rn$, so in light of~\eqref{eq:genus-bound}, we conclude $g_4(\#^n K_r)=rn$. Since $\Upsilon_{K_r}(t)=-\tau\cdot t$ near $t=0$ \cite{OSSUpsilon}*{Proposition~1.6}, we know that
\[
\max_{t\in [0,1]} \Upsilon_{\#^nK_r}(t)/t=n\max_{t\in [0,1]} \Upsilon_{K_r}(t)/t\ge nr,
\]
and~\eqref{eq:tau-upsilon-Kr} implies that
 \[
\min_{t\in [0,1]} \Upsilon_{\#^n K_r}(t)/t=n\min_{t\in [0,1]} \Upsilon_{K_r}(t)/t\le -nr.
\]
Applying Theorem~\ref{thm:Upsilon-bound}, the proof is complete.
\end{proof}

\begin{rem}
 The proof also shows that $\nu^+(\#^n K_{r})=\nu^+(-\#^n K_{r})=nr$. The bounds on $\Upsilon$ imply that $\nu^+(\#^n K_r)$ and $\nu^+(-\#^n K_r)$ are both at least $nr$, while the fact that $g_4(\#^n K_r)=nr$ implies that they are at most $nr$, since $\nu^+(K)\le g_4(K)$, for any knot $K$.
\end{rem}

The graphs of $\Upsilon_{K_1}(t)$ and $\Upsilon_{K_1}(t)/t$ are shown in Figure~\ref{fig:Upsilon}.

 \begin{figure}[H]
	\centering
	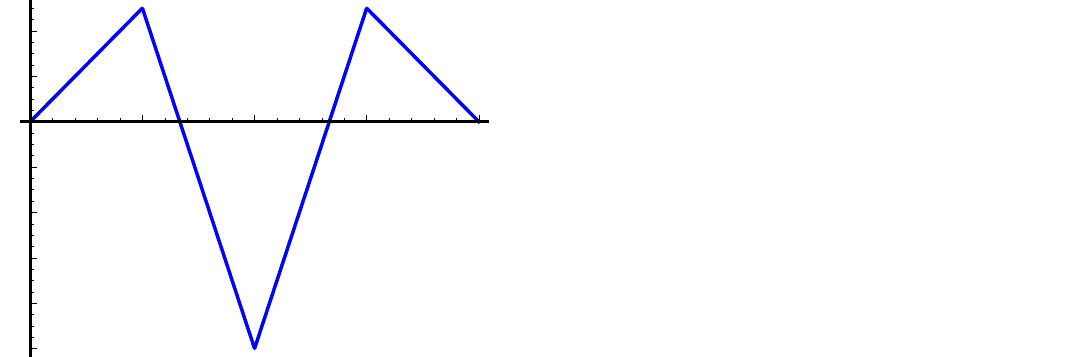
	\caption{The graphs of $\Upsilon_{K_1}(t)$ (left) and $\Upsilon_{K_1}(t)/t$ (right). Here $K_1=T_{2,11}\#-T_{4,5}$.}\label{fig:Upsilon}
\end{figure}

\section{Background on the link Floer TQFT}

In this section, we recall some background about knot Floer homology and the link Floer TQFT.

\subsection{Knot Floer homology}

Let $\bF[\scU,\scV]$ denote a 2-variable polynomial ring. Suppose that $\bK=(K,w,z)$ is a doubly pointed knot in a closed, oriented 3-manifold $Y$. If $\cH=(\Sigma,\as,\bs,w,z)$ is a Heegaard diagram of $(Y,\bK)$ and $\frs\in \Spin^c(Y)$, we let $\cCFK^-(Y,\bK,\frs)$ be the free $\bF[\scU,\scV]$-module generated by intersection points $\xs\in \bT_\a\cap \bT_{\b}$ satisfying $\frs_w(\xs)=\frs$. Here, 
\[
\bT_{\a}:=\a_1\times \dots\times \a_g \quad \text{and} \quad \bT_{\b}:=\b_1\times \cdots \times \b_g
\]
are two half-dimensional tori in the $g$-fold symmetric product $\Sym^g(\Sigma)$. We define a differential, which counts pseudo-holomorphic disks, via the formula
\[
\d\xs:=\sum_{\substack{\phi\in \pi_2(\xs,\ys) \\ \mu(\phi)=1}}\# (\cM(\phi)/\R) \scU^{n_w(\phi)} \scV^{n_z(\phi)}\cdot \ys,
\]
extended equivariantly over the action of $\bF[\scU,\scV]$. Here, $\cM(\phi)$ denotes the moduli space of pseudo-holomorphic disks in $\Sym^g(\Sigma)$, with boundary on $\bT_{\a}$ and $\bT_{\b}$, representing the class $\phi$. Also, $n_w(\phi)$ and $n_z(\phi)$ denote the multiplicities of the class $\phi$ on $w$ and $z$, respectively.

 There is also an infinity version $\cCFK^\infty(Y,\bK,\frs)$, which has the same generators, but is defined over the ring $\bF[\scU,\scV,\scU^{-1},\scV^{-1}]$. When it causes no confusion, we will write $\cCFK^-(K)$ instead of $\cCFK^-(Y,\bK,\frs)$. We write $\cHFK^-(K)$ for the homology.

Knot Floer homology admits a rich grading structure, which can be formulated in several ways. We will focus on the description which is considered in \cite{ZemAbsoluteGradings}, since it is the most natural from the point of view of link cobordism maps. If $\frs$ is torsion and $K$ is null-homologous, then $\cCFK^-(Y,\bK,\frs)$ admits two Maslov gradings $\gr_{\ws}$ and $\gr_{\zs}$. The variables $\scU$ and $\scV$ have $(\gr_{\ws},\gr_{\zs})$-bigradings $(-2,0)$ and $(0,-2)$, respectively. The Alexander grading $A$ satisfies
\begin{equation}
A=\frac{1}{2}(\gr_{\ws}-\gr_{\zs}). \label{eq:Alexander-grading-def}
\end{equation}
It is convenient also to consider the $\delta$-grading:
\[
\delta:=\frac{1}{2}(\gr_{\ws}+\gr_{\zs}).
\]

If $K$ is a knot in $S^3$, we  write $\scA_s(K)$  for the subset of $\cCFK^-(K)$ in Alexander grading $s$. The complex $\scA_s(K)$ is a module over the polynomial ring $\bF[U]$, where $U=\scU\scV$. 

There is a more common formulation of the full knot Floer complex, denoted $\CFK^\infty(K)$, which has generators $[\xs,i,j]$ such that $A(\xs)-j+i=0$. This is isomorphic to the subcomplex of $\cCFK^\infty(K)$ which lies in Alexander grading 0, via the map $\scU^i\scV^j \cdot \xs\mapsto [\xs,-i,-j]$. The complex $\scA_s(K)$ is also canonically isomorphic to the more standard large surgery complex $A_s(K)$, which is generated by $[\xs,i,j]$ satisfying $A(\xs)+i-j=0$ and $j\le s$, via the map $\scU^i \scV^j \cdot \xs\mapsto [\xs,-i,-j+s]$. As a $\delta$-graded complex, we have
\begin{equation}
\scA_s(K)\iso A_s(K)[0,-2s], \label{eq:A-isomorphism-with-shifts}
\end{equation}
where $[0,-2s]$ denotes shifting each bigrading by $[0,-2s]$.

It is easy to see that $\CFK^\infty(K)$ and $\cCFK^-(K)$ contain equivalent information. Our reason for using $\cCFK^-(K)$ is that it is a more natural packaging from the perspective of the cobordism maps.

More generally, if $\bL=(L,\ws,\zs)$ is an oriented, multi-pointed link in $Y$ (cf.~Definition~\ref{def:category}), the above construction adapts to give a link Floer complex over the ring $\bF[\scU,\scV]$. We denote this complex as $\cCFL^-(Y,\bL,\frs)$. This is a variation on the original link Floer complex defined by Ozsv\'{a}th and Szab\'{o} \cite{OSLinks}.

\subsection{Link Floer homology as a TQFT}
\label{subsec:TQFT}
Link Floer homology has a functorial framework which refines the original cobordism invariants of Ozsv\'{a}th and Szab\'{o} \cite{OSTriangles}. We recall the relevant cobordism category from \cite{JCob}:

\begin{define} \label{def:category}
\begin{enumerate} 
\item The objects are pairs $(Y,\bL)$, where $Y$ is a closed, oriented 3-manifold, and $\bL=(L,\ws,\zs)$ is an oriented link with two collections of basepoints, which alternate as one traverses the link. Furthermore, each component of $Y$ contains at least one component of $L$, and each component of $L$ contains at least two basepoints.
\item A morphism from $(Y_1,\bL_1)$ to $(Y_2,\bL_2)$ consists of a pair $(W,\cF)$, as follows: $W$ is a compact, oriented 4-manifold such that $\d W=-Y_1\sqcup Y_2$. Furthermore, $\cF$ consists of a properly embedded, oriented surface $S$ such that $\d S=-L_1\cup L_2$, which is decorated by a properly embedded 1-manifold $\cD\subset S$, such that $S\setminus \cD$ consists of two subsurfaces, $\cF_{\ws}$ and $\cF_{\zs}$, which meet along $\cD$. Furthermore, $\cD$ is disjoint from the basepoints, each component of $L_i\setminus \cD$ contains exactly one basepoint, and $\ws_i\subset \cF_{\ws}$ and $\zs_i\subset \cF_{\zs}$, for $i\in \{1,2\}$.
\end{enumerate}
\end{define}

\begin{rem}
In figures, we follow the convention that the $\ws$-subregion is shaded, and the $\ws$-basepoints are solid dots. The $\zs$-subregion is unshaded, and the $\zs$-basepoints are open dots.
\end{rem}

The first author associated functorial cobordism maps for decorated link cobordisms on the hat version of link Floer homology \cite{JCob} (the version obtained by setting all variables to zero). The construction used the contact gluing map of Honda, Kazez, and Mati\'{c} \cite{HKMTQFT}. The second author extended this to the full minus version, using maps associated to elementary cobordisms \cite{ZemCFLTQFT} (compare also \cite{AE}). The equivalence of the maps in \cite{JCob} and \cite{ZemCFLTQFT} is proven in \cite{JuhaszZemkeContactHandles}.

Following \cite{ZemCFLTQFT}, if $(W,\cF)$ is a decorated link cobordism from $(Y_1,\bL_1)$ to $(Y_2,\bL_2)$ and $\frs\in \Spin^c(W)$ is a $\Spin^c$ structure, there is an induced chain map
\[
F_{W,\cF,\frs}\colon \cCFL^-(Y_1,\bL_1,\frs|_{Y_1})\to \cCFL^-(Y_2,\bL_2,\frs|_{Y_2}),
\]
which is equivariant with respect to the action of $\bF[\scU,\scV]$. The grading changes of the cobordism maps are computed in \cite{ZemAbsoluteGradings}*{Theorem~1.4}. If $\frs|_{Y_1}$ and $\frs|_{Y_2}$ are torsion, then
\begin{equation}
\gr_{\ws}(F_{W,\cF,\frs}(\xs))-\gr_{\ws}(\xs)=\frac{c_1(\frs)^2-2\chi(W)-3\sigma(W)}{4}+\tilde{\chi}(\cF_{\ws}),
\label{eq:gr-w-grading}
\end{equation}
where $\tilde{\chi}(\cF_{\ws}):=\chi(\cF_{\ws})-\frac{1}{2}(|\ws_1|+|\ws_2|)$. Similarly, if $\frs|_{Y_1}-\PD[L_1]$ and $\frs|_{Y_2}-\PD[L_2]$ are torsion, then
\begin{equation}
\gr_{\zs}(F_{W,\cF,\frs}(\xs))-\gr_{\zs}(\xs)=\frac{c_1(\frs-\PD[\Sigma])^2-2\chi(W)-3\sigma(W)}{4}+\tilde{\chi}(\cF_{\zs}).
\label{eq:gr-z-grading}
\end{equation}
Similar to~\eqref{eq:Alexander-grading-def}, if $L_1$ and $L_2$ are null-homologous, and $\frs$ has torsion restriction to $Y_1$ and $Y_2$, then
\begin{equation}
A(F_{W,\cF,\frs}(\xs))-A(\xs)=\frac{\langle c_1(\frs), [\hat{S}]\rangle -[\hat{S}]\cdot [\hat{S}]}{2}+\frac{\chi(\cF_{\ws})-\chi(\cF_{\zs})}{2},
\label{eq:Alexander-grading}
\end{equation}
where $\hat{S}$ is obtained by capping $S$ with Seifert surfaces.

\subsection{A few helpful properties of the TQFT}

In this section, we recall several basic properties of the link Floer TQFT which we will use in this paper.

We first recall the \emph{bypass relation}, which is inspired by contact geometry (cf. \cite{HKMTQFT}*{Lemma~7.4}). If $\cF_1$, $\cF_2$, and $\cF_3$ are decorations of a fixed link cobordism $S\subset W$, which coincide outside of a disk $D$, and inside $D$ the dividing sets have the form shown in Figure~\ref{fig:15}, then
\begin{equation}
F_{W,\cF_1,\frs}+F_{W,\cF_2,\frs}+F_{W,\cF_3,\frs}\simeq 0.
\end{equation}
See \cite{ZemConnectedSums}*{Lemma~1.4}.

 \begin{figure}[H]
	\centering
	%% Creator: Inkscape 1.0 (4035a4fb49, 2020-05-01), www.inkscape.org
%% PDF/EPS/PS + LaTeX output extension by Johan Engelen, 2010
%% Accompanies image file 'fig15.pdf' (pdf, eps, ps)
%%
%% To include the image in your LaTeX document, write
%%   \input{<filename>.pdf_tex}
%%  instead of
%%   \includegraphics{<filename>.pdf}
%% To scale the image, write
%%   \def\svgwidth{<desired width>}
%%   \input{<filename>.pdf_tex}
%%  instead of
%%   \includegraphics[width=<desired width>]{<filename>.pdf}
%%
%% Images with a different path to the parent latex file can
%% be accessed with the `import' package (which may need to be
%% installed) using
%%   \usepackage{import}
%% in the preamble, and then including the image with
%%   \import{<path to file>}{<filename>.pdf_tex}
%% Alternatively, one can specify
%%   \graphicspath{{<path to file>/}}
%% 
%% For more information, please see info/svg-inkscape on CTAN:
%%   http://tug.ctan.org/tex-archive/info/svg-inkscape
%%
\begingroup%
  \makeatletter%
  \providecommand\color[2][]{%
    \errmessage{(Inkscape) Color is used for the text in Inkscape, but the package 'color.sty' is not loaded}%
    \renewcommand\color[2][]{}%
  }%
  \providecommand\transparent[1]{%
    \errmessage{(Inkscape) Transparency is used (non-zero) for the text in Inkscape, but the package 'transparent.sty' is not loaded}%
    \renewcommand\transparent[1]{}%
  }%
  \providecommand\rotatebox[2]{#2}%
  \newcommand*\fsize{\dimexpr\f@size pt\relax}%
  \newcommand*\lineheight[1]{\fontsize{\fsize}{#1\fsize}\selectfont}%
  \ifx\svgwidth\undefined%
    \setlength{\unitlength}{207.12273705bp}%
    \ifx\svgscale\undefined%
      \relax%
    \else%
      \setlength{\unitlength}{\unitlength * \real{\svgscale}}%
    \fi%
  \else%
    \setlength{\unitlength}{\svgwidth}%
  \fi%
  \global\let\svgwidth\undefined%
  \global\let\svgscale\undefined%
  \makeatother%
  \begin{picture}(1,0.28915228)%
    \lineheight{1}%
    \setlength\tabcolsep{0pt}%
    \put(0,0){\includegraphics[width=\unitlength,page=1]{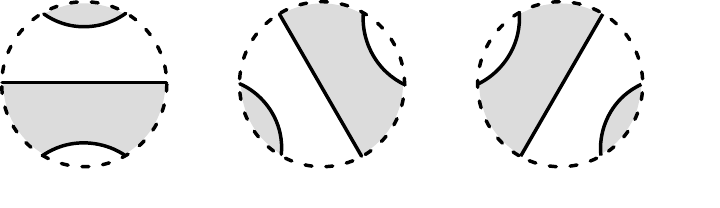}}%
    \put(0.11714968,0.00577573){\color[rgb]{0,0,0}\makebox(0,0)[t]{\lineheight{0}\smash{\begin{tabular}[t]{c}$\cF_1$\end{tabular}}}}%
    \put(0.44777536,0.00577573){\color[rgb]{0,0,0}\makebox(0,0)[t]{\lineheight{0}\smash{\begin{tabular}[t]{c}$\cF_2$\end{tabular}}}}%
    \put(0.77895959,0.00577573){\color[rgb]{0,0,0}\makebox(0,0)[t]{\lineheight{0}\smash{\begin{tabular}[t]{c}$\cF_3$\end{tabular}}}}%
    \put(0.27875024,0.15647277){\color[rgb]{0,0,0}\makebox(0,0)[t]{\lineheight{0}\smash{\begin{tabular}[t]{c} $+$\end{tabular}}}}%
    \put(0.60999994,0.15647277){\color[rgb]{0,0,0}\makebox(0,0)[t]{\lineheight{0}\smash{\begin{tabular}[t]{c}$+$\end{tabular}}}}%
    \put(0.92557801,0.15647272){\color[rgb]{0,0,0}\makebox(0,0)[lt]{\lineheight{0}\smash{\begin{tabular}[t]{l}$\simeq $\end{tabular}}}}%
    \put(0.97613765,0.15647277){\color[rgb]{0,0,0}\makebox(0,0)[lt]{\lineheight{0}\smash{\begin{tabular}[t]{l}$0$\end{tabular}}}}%
  \end{picture}%
\endgroup%

	\caption{The bypass relation.}\label{fig:15}
\end{figure}

Another property that we will use is an interpretation of certain dividing sets on cylinders as the following basepoint actions on knot Floer homology. If $\bL=(L,\ws,\zs)$ is a link in $Y$ and $w\in \ws$, there is an endomorphism of link Floer homology
\[
\Phi_{w}\colon \cCFL^-(Y,\bL,\frs)\to \cCFL^-(Y,\bL,\frs)
\]
given by the formula
\[
\Phi_{w}(\xs)=\scU^{-1}\sum_{\substack{\phi\in \pi_2(\xs,\ys) \\ \mu(\phi)=1}}n_w(\phi)\# (\cM(\phi)/\R) \scU^{n_{\ws}(\phi)} \scV^{n_{\zs}(\phi)}\cdot \ys.
\]
There is a similar endomorphism for the basepoint $z$, denoted $\Psi_{z}$. These endomorphisms have been studied by Sarkar and the first author. See \cite{SarkarMovingBasepoints} \cite{ZemQuasi} \cite{ZemCFLTQFT}*{Section~4.2}.

The second author proved \cite{ZemConnectedSums}*{Lemma~4.1} that the maps $\Phi_w$ and $\Psi_z$ coincide with the maps induced by the decorated link cobordisms shown in Figure~\ref{fig:22}. (Note that therein, there are just two basepoints on the corresponding component of $\bL$, which is the only case we need).

\begin{figure}[H]
	\centering
	%% Creator: Inkscape 1.0 (4035a4fb49, 2020-05-01), www.inkscape.org
%% PDF/EPS/PS + LaTeX output extension by Johan Engelen, 2010
%% Accompanies image file 'fig22.pdf' (pdf, eps, ps)
%%
%% To include the image in your LaTeX document, write
%%   \input{<filename>.pdf_tex}
%%  instead of
%%   \includegraphics{<filename>.pdf}
%% To scale the image, write
%%   \def\svgwidth{<desired width>}
%%   \input{<filename>.pdf_tex}
%%  instead of
%%   \includegraphics[width=<desired width>]{<filename>.pdf}
%%
%% Images with a different path to the parent latex file can
%% be accessed with the `import' package (which may need to be
%% installed) using
%%   \usepackage{import}
%% in the preamble, and then including the image with
%%   \import{<path to file>}{<filename>.pdf_tex}
%% Alternatively, one can specify
%%   \graphicspath{{<path to file>/}}
%% 
%% For more information, please see info/svg-inkscape on CTAN:
%%   http://tug.ctan.org/tex-archive/info/svg-inkscape
%%
\begingroup%
  \makeatletter%
  \providecommand\color[2][]{%
    \errmessage{(Inkscape) Color is used for the text in Inkscape, but the package 'color.sty' is not loaded}%
    \renewcommand\color[2][]{}%
  }%
  \providecommand\transparent[1]{%
    \errmessage{(Inkscape) Transparency is used (non-zero) for the text in Inkscape, but the package 'transparent.sty' is not loaded}%
    \renewcommand\transparent[1]{}%
  }%
  \providecommand\rotatebox[2]{#2}%
  \newcommand*\fsize{\dimexpr\f@size pt\relax}%
  \newcommand*\lineheight[1]{\fontsize{\fsize}{#1\fsize}\selectfont}%
  \ifx\svgwidth\undefined%
    \setlength{\unitlength}{256.04011898bp}%
    \ifx\svgscale\undefined%
      \relax%
    \else%
      \setlength{\unitlength}{\unitlength * \real{\svgscale}}%
    \fi%
  \else%
    \setlength{\unitlength}{\svgwidth}%
  \fi%
  \global\let\svgwidth\undefined%
  \global\let\svgscale\undefined%
  \makeatother%
  \begin{picture}(1,0.22839567)%
    \lineheight{1}%
    \setlength\tabcolsep{0pt}%
    \put(0,0){\includegraphics[width=\unitlength,page=1]{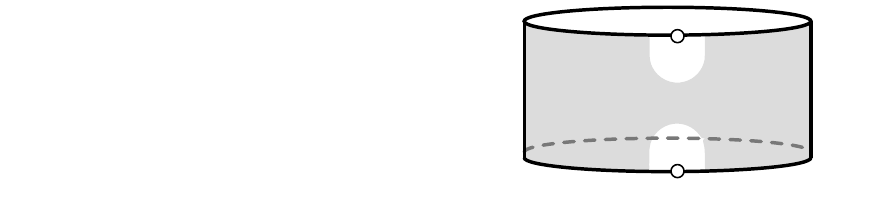}}%
    \put(0.95049238,0.1199753){\color[rgb]{0,0,0}\makebox(0,0)[lt]{\lineheight{0}\smash{\begin{tabular}[t]{l}$\Psi_z$\end{tabular}}}}%
    \put(0.73223336,0.00586148){\color[rgb]{0,0,0}\makebox(0,0)[lt]{\lineheight{0}\smash{\begin{tabular}[t]{l}$z$\end{tabular}}}}%
    \put(0,0){\includegraphics[width=\unitlength,page=2]{fig22.pdf}}%
    \put(0.35911084,0.11853848){\color[rgb]{0,0,0}\makebox(0,0)[lt]{\lineheight{0}\smash{\begin{tabular}[t]{l}$\Phi_w$\end{tabular}}}}%
    \put(0.13409411,0.00552963){\color[rgb]{0,0,0}\makebox(0,0)[lt]{\lineheight{0}\smash{\begin{tabular}[t]{l}$w$\end{tabular}}}}%
    \put(0,0){\includegraphics[width=\unitlength,page=3]{fig22.pdf}}%
  \end{picture}%
\endgroup%

	\caption{The basepoint actions as dividing sets on a cylinder.}\label{fig:22}
\end{figure}

A final relation involves the effect of stabilizing a surface by adding a tube (i.e., attaching the boundary of a 3-dimensional 1-handle in $W$, which is disjoint from $\cF$ except at its feet):

\begin{lem}\label{lem:stabilization}
Suppose that $(W,\cF)$ is a link cobordism, and $\cF'$ is obtained from $\cF$ by a 1-handle stabilization. If the feet of the stabilizing 1-handle are contained entirely in the $\ws$ subregion of $\cF$ and the dividing arcs of $\cF'$ coincide with those of $\cF$, then
\[
F_{W,\cF',\frs}=\scU \cdot F_{W,\cF,\frs}.
\]
If the feet of the stabilizing 1-handle are contained in the $\zs$ subregion of $\cF$, then
\[
F_{W,\cF',\frs}=\scV\cdot F_{W,\cF,\frs}.
\]
\end{lem}

\begin{proof}
When there is one dividing arc, Lemma~\ref{lem:stabilization} follows from \cite{JZStabilization}*{Lemma~5.3}. A simple proof that works more generally is as follows. By factoring the cobordism map through a neighborhood of the tube, one may reduce to the case when $W=B^4$, and $\cF$ consists of a pair of boundary parallel disks, which bound a 2-component unlink in $S^3$, such that furthermore each disk has a single dividing arc. In this case, the cobordism maps for $\cF$ and $\cF'$ may be computed explicitly.  The Floer homology for a 2-component unlink is $\langle \theta^+,\theta^-\rangle \otimes \bF[\scU,\scV]$. The map for $\cF$ is given by $1\mapsto \theta^+$, extended $\bF[\scU,\scV]$-equivariantly. If the tube is added to the $\zs$ subregion, then the cobordism map for $\cF'$ is equal to the map for $\cF$, followed by two $\zs$ band attachments. By an easy model computation on a genus 0 diagram with 4 basepoints, the composition of these two $\zs$-bands is multiplication by $\scV$. A similar argument works for a tube added to the $\ws$ subregion. The diagram where the model computation may be performed is shown in \cite{ZemCFLTQFT}*{Figure~8.3}.
\end{proof}

\section{The invariants $Y_n(K)$}
\label{def:invariants-Y}
Let $\cS_n$ denote a standard staircase complex, with $n$ steps, each of length $n$, which we view as a complex over $\bF[\scU,\scV]$. The complex $\cS_n$ has $2n+1$ generators, which we denote by $y_{-n},y_{-n+1},\dots, y_{n-1}, y_{n}$. The differential is given by
\[
\d(y_{-n+2i+1})=\scU\cdot y_{-n+2i}+\scV\cdot y_{-n+2i+2} \quad \text{and} \quad \d(y_{-n+2i})=0,
\]
extended equivariantly over $\bF[\scU,\scV]$. The $(\gr_{\ws},\gr_{\zs})$-bigrading of $y_i$ is $(-n-i,-n+i)$. If $K$ is a knot in $S^3$, write $\cC_K$ for $\cCFK^-(K)$. We define
\[
Y_n(K):=V_0(\cC_K\otimes \cS_n^\vee).
\]
Here, $\cS_n^\vee$ denotes the dual $\bF[\scU,\scV]$-complex to $\cS_n$, and
\[
V_0(C)=-d(A_0(C))/2,
\]
where $A_0(C)$ denotes the subcomplex generated by monomials $\scU^i\scV^j\cdot \xs$ with $i,j\ge 0$ and $A(\xs)+j-i=0$.

Note that the dual complex $\cS_n^\vee$ has generators $x_{-n},\dots, x_n$ such that $x_i$ has bigrading  $(n+i,n-i)$.

\begin{lem}
 The invariants $Y_n(K)$ satisfy the following:
 \begin{enumerate}
  \item\label{claim:Y-0} $Y_n(K)$ is a concordance invariant.
 \item\label{claim:Y-1} $0\le Y_n(K)$.
 \item\label{claim:Y-2} $Y_{n}(K)+1\le Y_{n+1}(K)\le Y_{n}(K)$.
 \item\label{claim:Y-3} $Y_n(K)=0$ for sufficiently large $n$.
 \item\label{claim:Y-4} $V_n(K)\le Y_n(K)$.
 \end{enumerate}
\end{lem}
\begin{proof}
Claim~\eqref{claim:Y-0} follows since the invariant $V_0$ is a well-defined invariant of a local equivalence class of knot complexes, and the tensor product operation is a well-defined group operation on the set of local equivalence classes; see \cite{ZemConnectedSums}*{Proposition~2.6}  (forgetting about the involutive part), \cite{HomSurvey} or \cite{KimParkInfiniteRank}*{Section~3}.

Claim~\eqref{claim:Y-1} follows since there is a grading preserving inclusion of $A_0(\cC_K\otimes \cS^\vee_n)$ into $B(\cC_K)\otimes_{\bF[U]} B(\cS_n^{\vee})$, where $B(\cC_K)$ denotes the subcomplex generated by monomials $\scU^{i}\scV^j\cdot \xs$ with $i\ge 0$, no restriction on $j$, and with $A(\xs)+j-i=0$, and similarly for $B(\cS_n^{\vee})$. However, $B(\cC_K)\simeq B(\cS_n^{\vee})\simeq \bF[U]$, as graded complexes, where we give $1$ grading 0. The claim follows.

We now consider the first inequality of claim~\eqref{claim:Y-2}. We define a chain map
\[
\Pi\colon \cS_{n+1}^{\vee}\to \cS_{n}^\vee
\]
in Figure~\ref{fig:the-map-pi-1}. The map $\Pi$ has $(\gr_{\ws},\gr_{\zs})$-bigrading $(-2,-2)$, and sends $\bF[U]$-non-torsion elements to $\bF[U]$-non-torsion elements. Since $\Pi$ preserves the Alexander grading, we obtain a local map
\[
\id\otimes \Pi\colon  A_0(\cC_K\otimes \cS_{n+1}^{\vee})\to A_0(\cC_K\otimes \cS_n^\vee),
\]
which lowers the homological grading by $2$. Hence
\[
Y_{n+1}(K)\ge Y_n(K)+1,
\]
giving the first inequality.

\begin{figure}[H]
\begin{tikzcd}[column sep={1cm,between origins},row sep=.8cm,labels=description]
&&& x_{-n+1}
&&
\cdots &&
x_{n-1}&
\\
&&x_{-n}\ar[ur,dashed, "\scU"]
&&
x_{-n+2}\ar[ul,dashed, "\scV"]\ar[ur,dashed, "\scU"]&\cdots&
x_{n-2}\ar[ur, dashed, "\scU"]\ar[ul,dashed, "\scV"]&&
x_{n}\ar[ul,dashed, "\scV"]
\\
&x_{-n}&& x_{-n+2}\ar[uu, "\scV"]
&&
\cdots &&
x_{n}\ar[uu, "\scV"]&
\\
x_{-n-1}
\ar[ur,dashed, "\scU"]
&&x_{-n+1}
\ar[ul,dashed, "\scV"]
\ar[uu, "\scV"]
\ar[ur,dashed, "\scU"]
&&
x_{-n+3}\ar[ul,dashed, "\scV"]\ar[ur, dashed, "\scU"]\ar[uu, "\scV"]&\cdots&
x_{n-1}\ar[uu, "\scV"]\ar[ur, dashed, "\scU"]\ar[ul,"\scV",dashed]&&
x_{n+1}\ar[uu, "\scV"]\ar[ul,dashed, "\scV"]
\end{tikzcd}
\caption{The map $\Pi\colon \cS_{n+1}^{\vee}\to \cS_{n}^\vee$. Solid arrows denote $\Pi$. Dashed arrows denote the internal differentials.}
\label{fig:the-map-pi-1}
\end{figure}

The second inequality of claim~\eqref{claim:Y-2} follows from the existence of a grading preserving local map
\[
\scI\colon \cS_n^{\vee}\to \cS_{n+1}^{\vee}.
\]
The map $\scI$ is shown in Figure~\ref{fig:the-map-I-1}. The map $\scI$ sends $\bF[U]$-non-torsion elements to $\bF[U]$-non-torsion elements, and preserves the Alexander grading. Hence, we obtain a grading preserving local map
\[
\id \otimes \scI\colon A_0(\cC_K\otimes \cS_{n}^{\vee})\to A_0(\cC_K\otimes \cS_{n+1}^\vee),
\]
which implies that
\[
Y_{n+1}(K)\le Y_n(K).
\]

\begin{figure}[H]
\begin{tikzcd}[column sep={1cm,between origins},row sep=.8cm,labels=description]
&x_{-n}
&& x_{-n+2}
&&
\cdots &&
x_{n}
&&\,&\,
\\
x_{-n-1}
	\ar[ur,dashed,"\scU"]
&&x_{-n+1}
	\ar[ur,dashed,"\scU"]
	\ar[ul,dashed, "\scV"]
&&
x_{-n+3}
	\ar[ul,dashed,"\scV"]
	\ar[ur,dashed,"\scU"]
&\cdots&
x_{n-1}
	\ar[ur,dashed,"\scU"]
	\ar[ul,dashed, "\scV"]
&&
x_{n+1}
	\ar[ul,dashed, "\scV"]
&&
\,
\\
&&& x_{-n+1}
	\ar[uu, "\scU"]
&&
\cdots &&
x_{n-1}
	\ar[uu,"\scU"]&
\\
&&x_{-n}
	\ar[ur,dashed,"\scU"]
	\ar[uull, "\scV"]
	\ar[uu,"\scU"]
&&
x_{-n}
	\ar[ul,dashed, "\scV"]
	\ar[ur, dashed,"\scU"]
	\ar[uu, "\scU"]
&\cdots&
x_{n-2}
	\ar[ur, dashed,"\scU"]
	\ar[ul,"\scV",dashed]
	\ar[uu, "\scU"]&&
x_{n}
	\ar[ul,"\scV",dashed]
	\ar[uu, "\scU"] 
\end{tikzcd}
\caption{The map $\mathscr{I}\colon \cS_{n}^\vee\to \cS_{n+1}^\vee$. Solid arrows denote $\mathscr{I}$. Dashed arrows denote the internal differentials.}
\label{fig:the-map-I-1}
\end{figure}

Next we consider claim~\eqref{claim:Y-3}. The equality  $Y_n(\cC_K\otimes \cS_{n}^\vee)=0$ is equivalent to the existence of $\bF[U]$-non-torsion cycles $z_{-n},z_{-n+2},\dots, z_n$ in $\scA_{-n},\scA_{-n+2},\dots, \scA_{n}$, all with $\delta$-grading $n$, such that
\begin{equation}
\scV\cdot [z_{i-1}]=\scU \cdot [z_{i+1}],
\label{eq:staircase-relation}
\end{equation}
whenever $i-n$ is odd and $-n+1\le i\le n-1$, where the brackets denote the induced element of homology. First, let $n$ be so large that $H_*(\scA_0)$ has rank 0 or 1 in each $\delta$-grading of $n$ or less. Furthermore, pick $n$ large enough that $\scA_n$ contains all $\xs \cdot \scU^i \scV^j$ with $A(\xs)+j-i=n$ and with $i\ge 0$, and furthermore so that $\scA_{-n}$ contains all $\xs \cdot \scU^i \scV^j$ with $A(\xs)+j-i=-n$ which have $j\ge 0$. In particular, there are $\bF[U]$-non-torsion cycles $z_{-n}\in \scA_{-n}$ and $z_{n}\in \scA_{n}$ of $\delta$-grading $-n$. By the condition on $\scA_0$, we know that $\scU^n\cdot[z_n]=\scV^n\cdot [z_{-n}]$. We note that the sequence
\[
\scV^n [z_n],\scU\scV^{n-1} [z_n],\dots, \scU^n [z_n]=\scV^n [z_{-n}], \scU \scV^{n-1} [z_{-n}],\dots, \scU^n[z_{-n}] 
\]
satisfies ~\eqref{eq:staircase-relation} and each element has $\delta$-grading $-2n$. Hence $Y_{2n}(K)=0$.

Finally, we consider claim~\eqref{claim:Y-4}. The generator $x_{-n}$ of $\cS_n^\vee$ has Alexander grading $-n$, so the complex $A_0(\cC_K\otimes \cS_{n}^\vee)$ has an $\bF[U]$-module summand equal to $\scA_{n}\otimes x_{-n}$. Furthermore, projection onto this summand gives a grading preserving chain map, which sends $\bF[U]$-non-torsion elements to $\bF[U]$-non-torsion elements. Hence
\[
Y_n(K)\ge V_n(K). \qedhere
\] 
\end{proof}

\section{Slice genus and clasp number bounds}
\label{sec:bounds-and-Y}

In this section, we prove Theorem~\ref{thm:Y-genus-clasp} of the introduction:

\begin{thm}
\label{thm:Y-genus-clasp-restated}
 Suppose $K$ is a knot in $S^3$.
\begin{enumerate}
\item\label{subthm:Y-genus} If $0\le n\le g_4(K)$, then
\[
Y_n(K)\le \left \lceil \frac{g_4(K)-n}{2} \right\rceil,
\]
and $Y_n(K) = 0$ if $n > g_4(K)$.
\item\label{subthm:Y-clasp} If $0\le n\le c_4^+(K)$, then
\[
Y_n(K)\le \left \lceil \frac{c_4^+(K)-n}{2}\right \rceil,
\]
and $Y_n(K) = 0$ if $n > c_4^+(K)$.
\end{enumerate}
If $n\ge g_4(K)$ or $n\ge c^+(K)$, then $Y_n(K)=0$.
\end{thm}

\begin{lem} \label{lem:almost-trivial-if-loop}
Suppose that $(W,\cF)\colon (Y_1,\bL_1)\to (Y_2,\bL_2)$ is a link cobordism such that $b_1(W)=0$, and the dividing set of $\cF$ contains a closed loop. Then
\[
\scU\cdot F_{W,\cF,\frs}\simeq 0 \quad  \text{and} \quad \scV\cdot F_{W,\cF,\frs}\simeq 0.
\] 
\end{lem}
\begin{proof} Suppose there is a closed loop in the dividing set. We factor the cobordism map through a neighborhood of the closed curve. After an isotopy of the dividing set, the map factors through the based link $(S^1\times S^2, \bL)$, where $\bL$ consists of two $S^1$-fibers of $S^1\times S^2$, oriented in opposite directions. Furthermore $\bL$ has two basepoints on each link component. Using the interpretation of the basepoint actions as the decorated link cobordisms in Figure~\ref{fig:22}, we may additionally factor the cobordism map through the map $\Phi_{w_1}\Psi_{z_1}$, where $w_1$ and $z_1$ are the basepoints on one of the link components of $\bL$. See Figure~\ref{fig:12}.

\begin{figure}[H]
	\centering
	%% Creator: Inkscape 1.0 (4035a4fb49, 2020-05-01), www.inkscape.org
%% PDF/EPS/PS + LaTeX output extension by Johan Engelen, 2010
%% Accompanies image file 'fig12.pdf' (pdf, eps, ps)
%%
%% To include the image in your LaTeX document, write
%%   \input{<filename>.pdf_tex}
%%  instead of
%%   \includegraphics{<filename>.pdf}
%% To scale the image, write
%%   \def\svgwidth{<desired width>}
%%   \input{<filename>.pdf_tex}
%%  instead of
%%   \includegraphics[width=<desired width>]{<filename>.pdf}
%%
%% Images with a different path to the parent latex file can
%% be accessed with the `import' package (which may need to be
%% installed) using
%%   \usepackage{import}
%% in the preamble, and then including the image with
%%   \import{<path to file>}{<filename>.pdf_tex}
%% Alternatively, one can specify
%%   \graphicspath{{<path to file>/}}
%% 
%% For more information, please see info/svg-inkscape on CTAN:
%%   http://tug.ctan.org/tex-archive/info/svg-inkscape
%%
\begingroup%
  \makeatletter%
  \providecommand\color[2][]{%
    \errmessage{(Inkscape) Color is used for the text in Inkscape, but the package 'color.sty' is not loaded}%
    \renewcommand\color[2][]{}%
  }%
  \providecommand\transparent[1]{%
    \errmessage{(Inkscape) Transparency is used (non-zero) for the text in Inkscape, but the package 'transparent.sty' is not loaded}%
    \renewcommand\transparent[1]{}%
  }%
  \providecommand\rotatebox[2]{#2}%
  \newcommand*\fsize{\dimexpr\f@size pt\relax}%
  \newcommand*\lineheight[1]{\fontsize{\fsize}{#1\fsize}\selectfont}%
  \ifx\svgwidth\undefined%
    \setlength{\unitlength}{202.10639587bp}%
    \ifx\svgscale\undefined%
      \relax%
    \else%
      \setlength{\unitlength}{\unitlength * \real{\svgscale}}%
    \fi%
  \else%
    \setlength{\unitlength}{\svgwidth}%
  \fi%
  \global\let\svgwidth\undefined%
  \global\let\svgscale\undefined%
  \makeatother%
  \begin{picture}(1,0.33715017)%
    \lineheight{1}%
    \setlength\tabcolsep{0pt}%
    \put(0,0){\includegraphics[width=\unitlength,page=1]{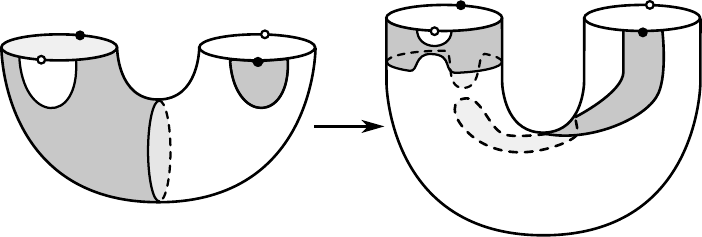}}%
    \put(0.02644108,0.31102162){\color[rgb]{0,0,0}\makebox(0,0)[lt]{\lineheight{1.25}\smash{\begin{tabular}[t]{l}$L_1$\end{tabular}}}}%
  \end{picture}%
\endgroup%

	\caption{Factoring the cobordism map in a neighborhood of a closed loop in the dividing set.}\label{fig:12}
\end{figure}

There is an action of $H_1(S^1\times S^2)$ on $\cCFL^-(S^1\times S^2,\bL)$, which extends the action of $H_1$ on the ordinary Heegaard Floer invariants \cite{OSDisks}*{Section~4.2.5}, as follows. We represent a class $\g\in H_1(S^1\times S^2)$ as an immersed curve on the Heegaard surface. Then $A_\g(\xs)$ is computed by counting holomorphic disks, weighted by the quantity $a(\g,\phi)$, which is the total change in multiplicities along $\g$ of the domain of $\phi$ across the alpha curves. We consider the case when $\g=[L_1]$. We may represent the component $L_1$ of $\bL$ also as an immersed curve on the Heegaard surface, which intersects only the alpha curves between the segment oriented from $z_1$ to $w_1$, and only the beta curves on the segment from $w_1$ to $z_1$. Hence, taking coefficients modulo 2, we obtain
\[
a(L_1,\phi)=n_{w_1}(\phi)+n_{z_1}(\phi),
\]
which implies the relation
\[
A_{[L_1]}=\scU \Phi_{w_1}+\scV \Psi_{z_1}.
\]
Compare \cite{ZemCFLTQFT}*{Lemma~14.12}.
Hence, we have
\[
\scU \Phi_{w_1}\Psi_{z_1}=(A_{[L_1]}+\scV \Psi_{z_1})\Psi_{z_1}\simeq A_{[L_1]} \Psi_{z_1},
\]
since $\Psi_{z_1}^2\simeq 0$.
In particular, the cobordism map $\scU\cdot F_{W,\cF,\frs}$ factors through the action of $[L_1]\in H_1(W)$. Since $b_1(W)=0$, any map with non-trivial twisting by $H_1(W)/\Tors$ will be null-homotopic, and hence
\[
\scU \cdot F_{W,\cF,\frs}\simeq 0.
\]
A similar argument implies $\scV\cdot F_{W,\cF,\frs}\simeq 0$.
\end{proof}

\begin{lem}\label{lem:dividing-euler-characteristic-1}
 Suppose that $S\subset B^4$ is a connected, oriented, and properly embedded surface with boundary $K\subset S^3$. Suppose that $\cF$ and $\cF'$ are two decorations of $S$, consisting of a single arc, such that $\chi(\cF_{\ws})=\chi(\cF_{\ws}')$ and $\chi(\cF_{\zs})=\chi(\cF_{\zs}')$. Then
 \[
\begin{split}
\left[\scU\cdot F_{B^4,\cF}(1)\right]&=\left[\scU \cdot F_{B^4, \cF'}(1)\right], \\
\left[\scV\cdot F_{B^4,\cF}(1)\right]&=\left[\scV \cdot F_{B^4, \cF'}(1)\right].
\end{split}
 \]
\end{lem}
\begin{proof}
Write $\cA$ for the dividing arc of $\cF$, and $\cA'$ for the dividing arc of $\cF'$. By the assumption on the Euler characteristics of the $\ws$ and $\zs$ regions, there is a diffeomorphism $\phi\colon S\to S$ which maps $\cA$ to $\cA'$. The mapping class group of $S$ is generated by non-separating Dehn twists along simple closed curves. Our proof will be to show that, after multiplying by $\scU$ or $\scV$, the cobordism maps are unchanged by replacing $\cA$ with $\phi(\cA)$, when $\phi$ is a Dehn twist along a non-separating curve.

By performing bypass relations parallel to $\g$, we may write $F_{B^4,\cF}$ as a sum of cobordism maps for dividing sets $\cF_1,\dots, \cF_n$ which each intersect $\g$ exactly 0 or 2 times. See Figure~\ref{fig:27}. Write $\cA_i$ for the dividing arcs of $\cF_i$ (which may no longer consist of a single arc). It is sufficient to show the main claim when $\cF$ consists of $S$ decorated with $\cA_i$, and $\cF'$ consists of $S$ decorated with  $\phi(\cA_i)$. We note that a bypass can be used to relate the maps induced by the dividing sets $\cA_i$ and $\phi(\cA_i)$. See Figure~\ref{fig:20}. The bypass relation involves a third curve, which has $\g$ itself as a component. This map induces the trivial map after multiplying by either $\scU$ or $\scV$, by Lemma~\ref{lem:almost-trivial-if-loop}, so the main claim follows.
 \end{proof}

\begin{figure}[H]
	\centering
	%% Creator: Inkscape 1.0 (4035a4fb49, 2020-05-01), www.inkscape.org
%% PDF/EPS/PS + LaTeX output extension by Johan Engelen, 2010
%% Accompanies image file 'fig27.pdf' (pdf, eps, ps)
%%
%% To include the image in your LaTeX document, write
%%   \input{<filename>.pdf_tex}
%%  instead of
%%   \includegraphics{<filename>.pdf}
%% To scale the image, write
%%   \def\svgwidth{<desired width>}
%%   \input{<filename>.pdf_tex}
%%  instead of
%%   \includegraphics[width=<desired width>]{<filename>.pdf}
%%
%% Images with a different path to the parent latex file can
%% be accessed with the `import' package (which may need to be
%% installed) using
%%   \usepackage{import}
%% in the preamble, and then including the image with
%%   \import{<path to file>}{<filename>.pdf_tex}
%% Alternatively, one can specify
%%   \graphicspath{{<path to file>/}}
%% 
%% For more information, please see info/svg-inkscape on CTAN:
%%   http://tug.ctan.org/tex-archive/info/svg-inkscape
%%
\begingroup%
  \makeatletter%
  \providecommand\color[2][]{%
    \errmessage{(Inkscape) Color is used for the text in Inkscape, but the package 'color.sty' is not loaded}%
    \renewcommand\color[2][]{}%
  }%
  \providecommand\transparent[1]{%
    \errmessage{(Inkscape) Transparency is used (non-zero) for the text in Inkscape, but the package 'transparent.sty' is not loaded}%
    \renewcommand\transparent[1]{}%
  }%
  \providecommand\rotatebox[2]{#2}%
  \newcommand*\fsize{\dimexpr\f@size pt\relax}%
  \newcommand*\lineheight[1]{\fontsize{\fsize}{#1\fsize}\selectfont}%
  \ifx\svgwidth\undefined%
    \setlength{\unitlength}{231.81775852bp}%
    \ifx\svgscale\undefined%
      \relax%
    \else%
      \setlength{\unitlength}{\unitlength * \real{\svgscale}}%
    \fi%
  \else%
    \setlength{\unitlength}{\svgwidth}%
  \fi%
  \global\let\svgwidth\undefined%
  \global\let\svgscale\undefined%
  \makeatother%
  \begin{picture}(1,0.27491003)%
    \lineheight{1}%
    \setlength\tabcolsep{0pt}%
    \put(0,0){\includegraphics[width=\unitlength,page=1]{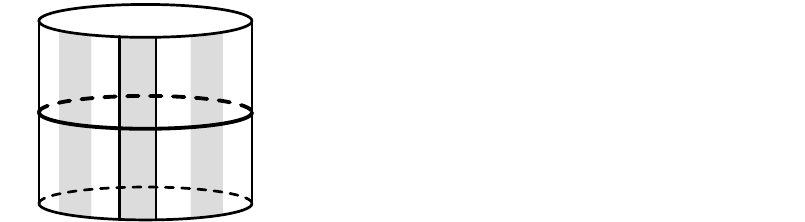}}%
    \put(0.03927011,0.13434154){\color[rgb]{0,0,0}\makebox(0,0)[rt]{\lineheight{1.25}\smash{\begin{tabular}[t]{r}$\g$\end{tabular}}}}%
    \put(0,0){\includegraphics[width=\unitlength,page=2]{fig27.pdf}}%
    \put(0.69366336,0.13245313){\color[rgb]{0,0,0}\makebox(0,0)[t]{\lineheight{1.25}\smash{\begin{tabular}[t]{c}$+$\end{tabular}}}}%
    \put(0.34854989,0.12002633){\color[rgb]{0,0,0}\makebox(0,0)[t]{\lineheight{1.25}\smash{\begin{tabular}[t]{c}$\simeq$\end{tabular}}}}%
    \put(0,0){\includegraphics[width=\unitlength,page=3]{fig27.pdf}}%
  \end{picture}%
\endgroup%

	\caption{Reducing the number of intersections of $\cA$ by using the bypass relation.}\label{fig:27}
\end{figure}

\begin{figure}[H]
	\centering
	%% Creator: Inkscape 1.0 (4035a4fb49, 2020-05-01), www.inkscape.org
%% PDF/EPS/PS + LaTeX output extension by Johan Engelen, 2010
%% Accompanies image file 'fig20.pdf' (pdf, eps, ps)
%%
%% To include the image in your LaTeX document, write
%%   \input{<filename>.pdf_tex}
%%  instead of
%%   \includegraphics{<filename>.pdf}
%% To scale the image, write
%%   \def\svgwidth{<desired width>}
%%   \input{<filename>.pdf_tex}
%%  instead of
%%   \includegraphics[width=<desired width>]{<filename>.pdf}
%%
%% Images with a different path to the parent latex file can
%% be accessed with the `import' package (which may need to be
%% installed) using
%%   \usepackage{import}
%% in the preamble, and then including the image with
%%   \import{<path to file>}{<filename>.pdf_tex}
%% Alternatively, one can specify
%%   \graphicspath{{<path to file>/}}
%% 
%% For more information, please see info/svg-inkscape on CTAN:
%%   http://tug.ctan.org/tex-archive/info/svg-inkscape
%%
\begingroup%
  \makeatletter%
  \providecommand\color[2][]{%
    \errmessage{(Inkscape) Color is used for the text in Inkscape, but the package 'color.sty' is not loaded}%
    \renewcommand\color[2][]{}%
  }%
  \providecommand\transparent[1]{%
    \errmessage{(Inkscape) Transparency is used (non-zero) for the text in Inkscape, but the package 'transparent.sty' is not loaded}%
    \renewcommand\transparent[1]{}%
  }%
  \providecommand\rotatebox[2]{#2}%
  \newcommand*\fsize{\dimexpr\f@size pt\relax}%
  \newcommand*\lineheight[1]{\fontsize{\fsize}{#1\fsize}\selectfont}%
  \ifx\svgwidth\undefined%
    \setlength{\unitlength}{221.00710424bp}%
    \ifx\svgscale\undefined%
      \relax%
    \else%
      \setlength{\unitlength}{\unitlength * \real{\svgscale}}%
    \fi%
  \else%
    \setlength{\unitlength}{\svgwidth}%
  \fi%
  \global\let\svgwidth\undefined%
  \global\let\svgscale\undefined%
  \makeatother%
  \begin{picture}(1,0.28835746)%
    \lineheight{1}%
    \setlength\tabcolsep{0pt}%
    \put(0,0){\includegraphics[width=\unitlength,page=1]{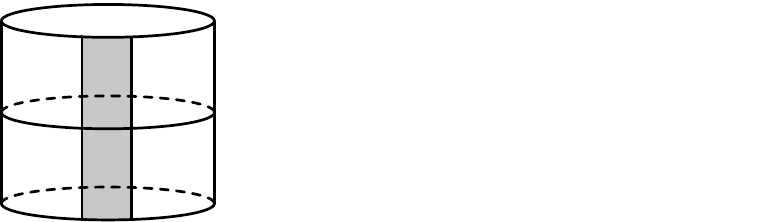}}%
    \put(0.03416612,0.09606957){\color[rgb]{0,0,0}\makebox(0,0)[lt]{\lineheight{1.25}\smash{\begin{tabular}[t]{l}$\g$\end{tabular}}}}%
    \put(0,0){\includegraphics[width=\unitlength,page=2]{fig20.pdf}}%
    \put(0.31668396,0.12567215){\color[rgb]{0,0,0}\makebox(0,0)[t]{\lineheight{1.25}\smash{\begin{tabular}[t]{c}$+$\end{tabular}}}}%
    \put(0.67867877,0.13915754){\color[rgb]{0,0,0}\makebox(0,0)[t]{\lineheight{1.25}\smash{\begin{tabular}[t]{c}$\simeq$\end{tabular}}}}%
  \end{picture}%
\endgroup%

	\caption{A bypass relation gives the effect of a Dehn twist along a curve $\g$ on the cobordism map.}\label{fig:20}
\end{figure}

\begin{lem}\label{lem:neighborhood-stabilization} Let $T\subset B^4$ be a properly embedded, genus 1 surface with boundary equal to an unknot $U$ in $S^3$. If $\cT$ is a decoration of $T$ by a dividing set $\cA$, then
\[
F_{B^4,\cT}=\begin{cases} \scU& \text{ if } |\cA|=1, \quad  g(\cT_{\ws})=1,\text{and} \quad g(\cT_{\zs})=0,\\
\scV& \text{ if }|\cA|=1,\quad g(\cT_{\zs})=1,\text{and} \quad g(\cT_{\ws})=0,\\
0& \text{ otherwise}.
\end{cases}
\]
\end{lem}

\begin{proof}The first two formulas follow  from \cite{ZemAbsoluteGradings}*{Theorem~1.8}, so we focus on the third. Consider the arc $a$ which has its endpoints on the unknot $U$. If $a$ bounds a boundary parallel bigon, and there is another component of the dividing set, then the induced cobordism map vanishes, since we may isotope the additional dividing curve towards the boundary parallel bigon, and factor the cobordism map through one of the endomorphisms $\Phi_w$ or $\Psi_z$, which clearly vanish on $\cCFK^-(S^3,U)$. See Figure~\ref{fig:22} for the interpretation of $\Phi_w$ and $\Psi_z$ as dividing sets.

The remaining case to consider is that $a$ does not bound a boundary parallel bigon, and instead represents a non-trivial element of $H_1(T,U)$. In this case, there must be at least one more curve in the dividing set, since $T\setminus a$ is connected. If there are any null-homotopic components of the dividing set, then the induced map is zero, by direct computation (see the definitions of the maps in \cite{ZemCFLTQFT}*{Section~4.1}). The remaining possibility is that there are an odd number of additional components of $\cA$, which are all parallel to $a$. In this case, $\cF_{\ws}$ and $\cF_{\zs}$ are both the disjoint union of equal numbers of annuli, so~\eqref{eq:gr-w-grading} and~\eqref{eq:gr-z-grading} imply that $F_{B^4,\cT}(1)$ has $(\gr_{\ws},\gr_{\zs})$-bigrading $(-1,-1)$. However the subset of $\cCFK^-(S^3,U)=\bF[\scU,\scV]$ in this grading is 0, so the induced element must be trivial.
\end{proof}

\begin{prop}
\label{prop:main-claim-dividing-set-ind}
Suppose that $W$ is a compact, oriented 4-manifold with boundary $S^3$, such that $b_1(W)=0$, and $S\subset W$ is a connected, properly embedded surface with boundary $K$, such that $S$ is the stabilization of a surface $S_0$ (i.e. $S=S_0\# T$ for an unknotted torus $T$). If $\cF$ and $\cF'$ are two decorations of $S$, both consisting of a single arc such that $\chi(\cF_{\ws})=\chi( \cF_{\ws}')$ and $\chi(\cF_{\zs})=\chi(\cF_{\zs}')$, then
 \[
 [F_{W,\cF,\frs}(1)]=[F_{W,\cF',\frs}(1)],
 \]
 for any $\frs\in \Spin^c(W)$.
\end{prop}
\begin{proof}
Our proof will be to show the cobordism maps for $S$ are unchanged by applying a Dehn twist to the dividing set along a non-separating curve on $S$. We may restrict our attention to Dehn twists along one of the $3g-1$ Lickorish generators of the mapping class group. We  assume that the first three generators, $c_1,$ $c_2$, and $c_3$, are chosen to intersect the stabilization torus $T$, while all the other generators are disjoint. See Figure~\ref{fig:21}.

\begin{figure}[H]
	\centering
	%% Creator: Inkscape 1.0 (4035a4fb49, 2020-05-01), www.inkscape.org
%% PDF/EPS/PS + LaTeX output extension by Johan Engelen, 2010
%% Accompanies image file 'fig21.pdf' (pdf, eps, ps)
%%
%% To include the image in your LaTeX document, write
%%   \input{<filename>.pdf_tex}
%%  instead of
%%   \includegraphics{<filename>.pdf}
%% To scale the image, write
%%   \def\svgwidth{<desired width>}
%%   \input{<filename>.pdf_tex}
%%  instead of
%%   \includegraphics[width=<desired width>]{<filename>.pdf}
%%
%% Images with a different path to the parent latex file can
%% be accessed with the `import' package (which may need to be
%% installed) using
%%   \usepackage{import}
%% in the preamble, and then including the image with
%%   \import{<path to file>}{<filename>.pdf_tex}
%% Alternatively, one can specify
%%   \graphicspath{{<path to file>/}}
%% 
%% For more information, please see info/svg-inkscape on CTAN:
%%   http://tug.ctan.org/tex-archive/info/svg-inkscape
%%
\begingroup%
  \makeatletter%
  \providecommand\color[2][]{%
    \errmessage{(Inkscape) Color is used for the text in Inkscape, but the package 'color.sty' is not loaded}%
    \renewcommand\color[2][]{}%
  }%
  \providecommand\transparent[1]{%
    \errmessage{(Inkscape) Transparency is used (non-zero) for the text in Inkscape, but the package 'transparent.sty' is not loaded}%
    \renewcommand\transparent[1]{}%
  }%
  \providecommand\rotatebox[2]{#2}%
  \newcommand*\fsize{\dimexpr\f@size pt\relax}%
  \newcommand*\lineheight[1]{\fontsize{\fsize}{#1\fsize}\selectfont}%
  \ifx\svgwidth\undefined%
    \setlength{\unitlength}{259.50758359bp}%
    \ifx\svgscale\undefined%
      \relax%
    \else%
      \setlength{\unitlength}{\unitlength * \real{\svgscale}}%
    \fi%
  \else%
    \setlength{\unitlength}{\svgwidth}%
  \fi%
  \global\let\svgwidth\undefined%
  \global\let\svgscale\undefined%
  \makeatother%
  \begin{picture}(1,0.24088849)%
    \lineheight{1}%
    \setlength\tabcolsep{0pt}%
    \put(0,0){\includegraphics[width=\unitlength,page=1]{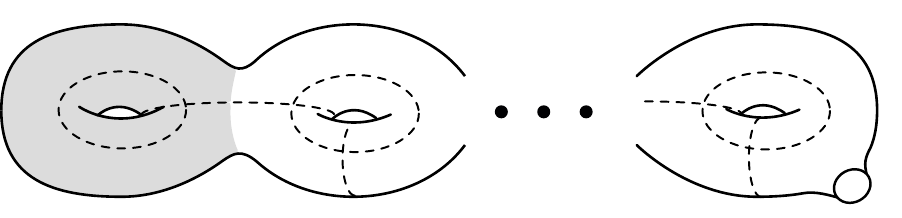}}%
    \put(0.13126583,0.04491574){\color[rgb]{0,0,0}\makebox(0,0)[lt]{\lineheight{1.25}\smash{\begin{tabular}[t]{l}$c_1$\end{tabular}}}}%
    \put(0.16944284,0.16314759){\color[rgb]{0,0,0}\makebox(0,0)[lt]{\lineheight{1.25}\smash{\begin{tabular}[t]{l}$c_2$\end{tabular}}}}%
    \put(0,0){\includegraphics[width=\unitlength,page=2]{fig21.pdf}}%
    \put(0.24295237,0.10515857){\color[rgb]{0,0,0}\makebox(0,0)[lt]{\lineheight{1.25}\smash{\begin{tabular}[t]{l}$c_3$\end{tabular}}}}%
    \put(0.97076798,0.00271698){\color[rgb]{0,0,0}\makebox(0,0)[lt]{\lineheight{1.25}\smash{\begin{tabular}[t]{l}$K$\end{tabular}}}}%
    \put(0.06700164,0.22902709){\color[rgb]{0,0,0}\makebox(0,0)[lt]{\lineheight{1.25}\smash{\begin{tabular}[t]{l}$T$\end{tabular}}}}%
  \end{picture}%
\endgroup%

	\caption{Lickorish generators of the mapping class group of a genus $g$ surface with one boundary component. The shaded region is the stabilization torus.}\label{fig:21}
\end{figure}

First, we consider Dehn twists along one of the curves $c_1$ and $c_2$. By using the bypass relation in the neck region of the stabilization, as in Figure~\ref{fig:27}, we may write $F_{W,\cF,\frs}$ as a sum of $n$ maps $F_{W,\cF_i,\frs}$, where $\cF_i$ has a dividing set which intersects the boundary of the stabilization torus in 0 or 2 points.  It is sufficient to show that each the maps $F_{W,\cF_i,\frs}$ is unchanged by a Dehn twist along $c_1$ or $c_2$. By factoring the cobordism map through a neighborhood of the stabilization, we may apply Lemma~\ref{lem:neighborhood-stabilization} to see that the map $F_{W,\cF_i,\frs}$ is zero unless $\cA_i\cap T$ is empty, or $\cA_i\cap T$ is a connected arc and one of $T\cap \cF_{i,\ws}$ and $T\cap \cF_{i,\zs}$ is a bigon, while the other is a genus 1 surface. In particular, we may assume that $\cF_i$ is isotopic to a dividing set which is disjoint from the stabilization torus. In particular, if a summand $F_{W,\cF_i,\frs}$ is non-zero, then the dividing set is fixed up to isotopy by Dehn twists along $c_1$ or $c_2$.

We now consider the remaining generators, $c_3,\dots, c_{3g-1}$, of the mapping class group. Using the previous argument, it is sufficient to show the claim for dividing sets $\cF_i$ which are disjoint from the stabilization region. We may view the dividing set $\cF_i$ as being obtained by attaching a 1-handle to a decoration $\cF_{i,0}$ of $S_0$, such that the feet are attached to either the $\ws$ subregion or the $\zs$ subregion. For the sake of demonstration, assume that this 1-handle stabilization occurs in the $\ws$ subregion. By Lemma~\ref{lem:stabilization},
\begin{equation}
F_{W,\cF_i,\frs}=\scU\cdot F_{W,\cF_{i,0},\frs},
\label{eq:destabilize-1}
\end{equation}
 Write $A\subset S$ for a regular neighborhood of the curve $c_1$. We view $A$ as the region added by a 1-handle stabilization, attached to $S_0$. Note that $A$ is disjoint from all of the mapping class group generators, except for $c_1$ and $c_2$. Suppose $\g$ is one of $c_3,\dots, c_{3g-1}$. Abusing notation, write $\phi$ for the Dehn twist along $\gamma$ on both $S$ and $S_0$. Since $\g$ is disjoint from $A$, we have that $F_{W,\phi(\cF_{i}),\frs}$ also has a 1-handle stabilization in the $\ws$ subregion, and hence
\begin{equation}
F_{W,\phi(\cF_i),\frs}=\scU\cdot F_{W,\phi(\cF_{i,0}),\frs}.
\label{eq:destabilize-2}
\end{equation}
 Using Lemma~\ref{lem:dividing-euler-characteristic-1}, we see that~\eqref{eq:destabilize-1} and~\eqref{eq:destabilize-2} imply the main claim.
\end{proof}

\subsection{Proof of Theorem~\ref{thm:Y-genus-clasp}}

We now prove our genus and clasp bounds for $Y_n(K)$:
\begin{proof}[Proof of Theorem~\ref{thm:Y-genus-clasp}]
We begin with part~\eqref{subthm:Y-genus}. Suppose that $0\le n\le g_4(K)$, and let $S\subset B^4$ be a slice surface for $K$, with genus $g$. We may assume that $g-n$ is even, since if it were odd, we could increase $g$ by 1 without changing $\lceil (g-n)/2\rceil$.

 Let $\cF_{-n},\cF_{-n+2},\dots,\cF_{n-2}, \cF_n$ denote decorations of $S$, each consisting of a single arc which divides $S$ into two connected subsurfaces. Furthermore,
we assume that 
\[
g(\cF_{i,\ws})=\frac{g-i}{2}\quad \text{and} \quad g(\cF_{i,\zs})=\frac{g+i}{2},
\]
which is possible to arrange, since $g$ and $n$ have the same parity. 
 Define
 \[
z_i=F_{B^4,\cF_{i}}(1), 
 \]
 for $i\in \{-n,-n+2,\dots, n-2,n\}$. Note that $z_i$ has bigrading $(-g+i,-g-i)$. In particular, $z_i$ has Alexander grading  $i$.
 
  Proposition~\ref{prop:main-claim-dividing-set-ind} implies that for all $i\in \{-n+1,-n+3,\dots, n-1\}$, there is a chain $z_i$, of homogeneous grading $(-g+i,-g-i)$, which satisfies
 \[
\d z_i= \scU\cdot z_{i+1}+\scV\cdot z_{i-1}.
 \]
  The element
 \[
y=\sum_{i=-n}^{n} z_i\otimes x_{-i} 
 \]
is a cycle in $\cC_K\otimes \cS_{n}^\vee$ which is $\bF[U]$-non-torsion, since each $x_{-n+2i}$ is $\bF[U]$-non-torsion. Furthermore, $y$ has $(\gr_{\ws},\gr_{\zs})$-bigrading $(n-g,n-g),$ since $x_i$ has grading $(n+i,n-i)$. In particular, $y\in \scA_0(\cC_K\otimes \cS_n^{\vee})$ since it has Alexander grading zero. As $y$ is $\bF[U]$ non-torsion, we have
\[
V_0(\cC_K\otimes \cS_n^{\vee})\le \frac{g-n}{2}=\left \lceil \frac{g-n}{2} \right \rceil,
\]
which completes the proof.

We now consider part~\eqref{subthm:Y-clasp}, concerning $c^+(K)$. Suppose that $S$ is a surface with $C^+$ positive double points, and $C^-$ negative double points. We blow up at the negative double points, and smooth the positive double points to obtain a genus $C^+$, null-homologous slice  surface $\hat{S}$ for $K$ in $B^4\# C^- \bar{\CP}^2$. As before, we may assume, by stabilizing once if necessary, that $C^+-n$ is even. We pick any $\Spin^c$ structure on $B^4\# C^-\bar{\CP}^2$ with maximal square, as well as decorations $\cF_{-n},\cF_{-n+2},\dots, \cF_{n}$ of $\hat{S}$, as above, and consider the elements $z_i=F_{B^4\# C^-\bar{\CP}^2, \cF_i, \frs}$. The same argument as for surfaces in $B^4$ shows that we obtain an $\bF[U]$-non-torsion cycle $y$ in $\cC_K\otimes \cS_n^\vee$ of bigrading $(n-C^+ ,n-C^+ )$, so we obtain the stated inequality involving $Y_n(K)$.
\end{proof}

\section{Involutive invariants}

In this Section, we review Hendricks and Manolescu's involutive Heegaard Floer homology, and prove Theorem~\ref{thm:involutive}
from the introduction:

\begin{thm} Suppose $K$ is a knot in $S^3$. Then the following hold:
\begin{equation}\label{thm:involutive-genus}
-\left\lceil \frac{g_4(K)+1}{2}\right\rceil \le \bar{V}_0(K)\le \ul{V}_0(K)\le \left \lceil \frac{g_4(K)+1}{2} \right \rceil,
\end{equation}
\begin{equation}\label{thm:involutive-clasp}
-\left\lceil \frac{c_4^-(K)+1}{2}\right\rceil \le \bar{V}_0(K)\le \ul{V}_0(K)\le \left \lceil \frac{c_4^+(K)+1}{2} \right \rceil.
\end{equation}
\end{thm}

\subsection{Background on involutive Heegaard Floer homology}

Involutive Heegaard Floer homology is a refinement of Heegaard Floer homology described by Hendricks and Manolescu \cite{HMInvolutive}. If $Y$ is a 3-manifold equipped with a self-conjugate $\Spin^c$ structure $\frs$, they constructed an $\bF[U,Q]/Q^2$-module $\HFI^-(Y,\frs)$. Their construction gives two involutive correction terms, $\bar{d}(Y,\frs)$ and $\underline{d}(Y,\frs)$, which satisfy
\[
\underline{d}(Y,\frs)\le d(Y,\frs)\le \bar{d}(Y,\frs),
\]
where $d$ is the ordinary correction term of Ozsv\'{a}th and Szab\'{o} \cite{OSIntersectionForms}.

For a knot $K$ in $S^3$, Hendricks and Manolescu defined two integer invariants $\bar{V}_0(K)$ and $\ul{V}_0(K)$, which satisfy
\[
\bar{V}_0(K)\le V_0(K)\le \ul{V}_0(K).
\]

 The invariants $\bar{V}_0(K)$ and $\ul{V}_0(K)$ may be computed directly from the knot Floer complex, as we now describe. Hendricks and Manolescu describe a knot involution
 \[
\iota_K\colon \CFK^\infty(K)\to \CFK^\infty(K). 
 \]
 For our present purposes, it is helpful to actually consider the knot involution $\iota_K$ as an endomorphism of the complex $\cCFK^-(K)$. On $\cCFK^-(K)$, the map $\iota_K$ satisfies
 \begin{equation}
\scU \iota_K=\iota_K \scV\quad \text{and} \quad \scV\iota_K=\iota_K \scU.\label{eq:skew-equivariant}
 \end{equation}
More generally, we say a map which satisfies~\eqref{eq:skew-equivariant} is \emph{skew-equivariant}. Furthermore, the map $\iota_K$ interchanges the gradings $\gr_{\ws}$ and $\gr_{\zs}$.

The subcomplex $A_0(K)\subset \CFK^\infty(K)$ is preserved by $\iota_K$. Hendricks and Manolescu define
\[
AI_0(K):=\Cone(A_0(K)\xrightarrow{Q(1+\iota_K)} QA_0(K)).
\]
We give $AI_0(K)$ the grading induced by $\CFK^\infty(K)$ and by setting $\deg(Q)=-1$. The invariants $\bar{V}_0(K)$ and $\ul{V}_0(K)$ are defined as
\[
\bar{V}_0(K):=-\frac{1}{2} \bar{d}(AI_0(K))\quad \text{and} \quad \ul{V}_0(K):=-\frac{1}{2} \ul{d}(AI_0(K)).
\]

\subsection{Involutive correction terms and the slice genus}

In this section, we prove equation~\eqref{thm:involutive-genus} of Theorem~\ref{thm:involutive}, namely that the involutive correction terms satisfy the stated slice genus bound. We prove equation~\eqref{thm:involutive-clasp} in the subsequent sections.

\begin{proof}[Proof of equation~\eqref{thm:involutive-genus} of Theorem~\ref{thm:involutive}]
Suppose that $S\subset B^4$ is an oriented slice surface in $B^4$ for the knot $K$. We may assume that $g(S)$ is odd, since if $g(S)$ is even, then we may stabilize $g(S)$ by taking the connected sum with an unknotted torus without increasing the quantity $\lceil (g(S)+1)/2\rceil$.

We let $S'$ be obtained by stabilizing $S$ once, taking the connected sum with an unknotted torus. Note that $g(S')$ is even. We pick a dividing set $\cA$ on $S'$, consisting of one arc which divides $S'$ into two components, both of which have the same genera. Furthermore, we may assume that the stabilization occurs entirely in the $\ws$ subregion. Write $\cF'$ for this dividing set on $S'$, and let $\cF$ denote the induced dividing set on $S$. Let $\bar{\cF}'$ denote the decorated surface obtained by reversing the roles of the $\ws$ and $\zs$ subregions, and adding a half-twist to the dividing set along $\d S'$. We define $\bar{\cF}$ similarly to be the corresponding decoration of $S$.

We consider the map
\[
F_{B^4,\cF'}\colon \bF[\scU,\scV]\to \cCFK^-(K),
\]
which preserves the Alexander grading by the grading change formulas of Section~\ref{subsec:TQFT}, since the genera of the $\ws$ and $\zs$ subregions coincide. It follows from \cite{ZemConnectedSums}*{Theorem~1.3} that
\begin{equation}
\iota_K\circ   F_{B^4,\cF'}\eqsim  F_{B^4,\bar{\cF}'} \circ \iota_0, \label{eq:involution-commutes}
\end{equation}
where $\iota_0\colon \bF[\scU,\scV]\to \bF[\scU,\scV]$ denotes the map which sends $\scU^i\scV^j$ to $\scU^j\scV^i$. Since the surface $S'$ is stabilized, and the genera of the $\ws$ and $\zs$ subregions of $\cF'$ and $\bar{\cF}'$ coincide, we conclude from Proposition~\ref{prop:main-claim-dividing-set-ind}  that
\[
F_{B^4,\cF'}\simeq  F_{B^4,\bar{\cF}'}.
\] 
 In particular, we may define a map
\[
F\colon \CFI^-(\emptyset)\to AI_0(K),
\]
which lowers the homological grading by $g(S)+1$. Concretely, if we view $AI_0(K)$ as consisting of the direct sum of two copies of $A_0(K)$, then the matrix for $F$ is given by the formula
\[
F=
\begin{pmatrix}
F_{B^4,\cF'}&0\\
H& F_{B^4,\cF'}
\end{pmatrix},
\]
where $H$ is a $\bF[\scU,\scV]$-skew equivariant, homogeneously $(-g(S)-1,-g(S)-1)$ graded homotopy between $F_{B^4,\cF'}\circ  \iota_0$ and $\iota_K \circ F_{B^4,\cF'}$. The map $F$ is clearly a chain map, and it follows from \cite{ZemAbsoluteGradings}*{Theorem~1.7} that the map $F$ becomes an isomorphism on homology after localizing at $U$. Hence, the odd and even towers of $H_*(AI_0(K))$ contain a $\bF[U]$-non-torsion element of grading $g(S)+1$ less than the generators of the odd and even towers of $\HFI^-(\emptyset)\iso \bF[U,Q]/Q^2$. Hence
\[
 \underline{d}(AI_0(K))\ge -g(S)-1,
\]
and so
\[
\ul{V}_0(K)\le \frac{g(S)+1}{2}=\left \lceil \frac{g(S)+1}{2}\right \rceil.
\]

To obtain the bound involving $\overline{V}_0(K)$, we turn around and reverse the orientation of $(B^4,S)$, to obtain a link cobordism from $(S^3,K)$ to $\emptyset$. As before, we assume that $g(S)$ is odd, and we stabilize once to obtain a surface $S'$. We write $\cG'$ for the same decoration of $S'$ as before, but given the opposite orientation as a surface. We obtain a map
\[
F_{B^4,\cG'}\colon \cCFK^-(K)\to \bF[\scU,\scV]
\]
inducing an Alexander grading change of 0. By the same reasoning as before, there is a homogeneously graded homotopy $J$ satisfying 
\[
F_{B^4,\cG'}\circ \iota_K+\iota_0 \circ F_{B^4,\cG'}=[\d,J],
\]
so we may build the chain map
\[
G=\begin{pmatrix}
F_{B^4,\cG'}&0\\
J& F_{B^4,\cG'}
\end{pmatrix}
\]
from $AI_0(K)$ to $\CFI^-(\emptyset)$. As before, the map $G$ becomes an isomorphism on homology after inverting $U$, and lowers the homological grading by $g(S)+1$. We conclude that the even (resp. odd) tower of $AI_0(K)$ has a generator which has grading at most $g(S)+1$ higher than the even (resp.~odd) generator of $\bF[U,Q]/Q^2=\HFI^-(\emptyset)$. In particular,
\[
\bar{d}(AI_0(K))\le g(S)+1,
\]
which immediately gives
\[
-\left \lceil \frac{ g(S)+1}{2} \right\rceil \le \bar{d}(AI_0(K)). \qedhere
\]

\end{proof}

\subsection{On the link Floer homology of the Hopf link}

There are two ways of orienting the components of the Hopf link which do not become equivalent after an overall change in orientation. We refer to the orientation which results in two negative crossings of the standard diagram as the \emph{negative} Hopf link, and we refer to the orientation which has two positive crossings as the \emph{positive} Hopf link. A Hopf link of sign $\varepsilon\in \{+,-\}$ bounds a normally immersed $D^2 \sqcup D^2$ with a single double point, which has sign $\varepsilon$.

%Thus far, we have focused mostly on the version of link Floer homology which has one variable $\scU$ for all of the $\ws$-basepoints, and one variable $\scV$ for all of the $\zs$-basepoints. It becomes convenient now to consider a refinement of this where we have one $\scU$ variable for each $\ws$-basepoint, and one $\scV$ variable for each $\zs$-basepoint. A diagram for the Hopf link must have (at least) four basepoints, so we now consider the link Floer complex over the ring $\bF[\scU_1,\scU_2,\scV_1,\scV_2]$.

 The complexes for the positive and negative Hopf links are shown in Figure~\ref{fig:Hopf}. The complexes may be computed by counting bigons on a simple Heegaard diagram (cf. \cite{MOIntegerSurgery}*{Figure~3}).

\begin{figure}[H]

\begin{tikzcd}[labels=description,row sep=1cm, column sep=1cm]
\ve{a}
&
\ve{b}
	\arrow[d,"\scV"]
	\arrow[l,"\scU"]	
	\\
\ve{c}
	\arrow[r, "\scV"]
	\arrow[u,"\scU"]
&\ve{d}
\end{tikzcd}
\qquad
\begin{tikzcd}[labels=description,row sep=1cm, column sep=1cm]
\ve{a}
&
\ve{b}
	\arrow[from=d,"\scV"]
	\arrow[from=l,"\scU"]	
	\\
\ve{c}
	\arrow[from=r, "\scV"]
	\arrow[from=u,"\scU"]
&\ve{d}.
\end{tikzcd}
\caption{The link Floer complexes of the positive (left) and negative (right) Hopf links.}
\label{fig:Hopf}
\end{figure}

%\begin{figure}[H]
%
%\begin{tikzcd}[labels=description,row sep=1cm, column sep=1cm]
%\ve{a}
%&
%\ve{b}
%	\arrow[d,"\scV_1"]
%	\arrow[l,"\scU_2"]	
%	\\
%\ve{c}
%	\arrow[r, "\scV_2"]
%	\arrow[u,"\scU_1"]
%&\ve{d}
%\end{tikzcd}
%\qquad
%\begin{tikzcd}[labels=description,row sep=1cm, column sep=1cm]
%\ve{a}
%&
%\ve{b}
%	\arrow[from=d,"\scV_1"]
%	\arrow[from=l,"\scU_2"]	
%	\\
%\ve{c}
%	\arrow[from=r, "\scV_2"]
%	\arrow[from=u,"\scU_1"]
%&\ve{d}.
%\end{tikzcd}
%\caption{The link Floer complexes of the positive (left) and negative (right) Hopf links.}
%\label{fig:Hopf}
%\end{figure}

We consider the neighborhood of a transverse double point. If the sign of the intersection point is negative, we may perform a blow-up and obtain a properly embedded surface in $B^4 \# \bar{\CP}^2$ that has relative homology class 
\[
0 \in H_2(B^4 \# \bar{\CP}^2, S^3), 
\]
and which topologically consists of two disks. The set of $\Spin^c$ structures on $\bar{\CP}^2$ can be identified with the odd integers, and the conjugation action sends $n$ to $-n$.  Since $\bar{\CP}^2$ is negative definite, there are two $\Spin^c$ structures whose Chern classes have maximal square, 
and these classes are switched by conjugation. Let $\cS$ denote this surface in $B^4\# \bar{\CP}^2$, decorated with two arcs which each divide a disk into two bigons.

\begin{lem}\label{lem:Hopf-CP2}
 View $(B^4\# \bar{\CP}^2,\cS)$ as a link cobordism from the empty set to the negative Hopf link, and let $\frs$ be a $\Spin^c$ structure on $B^4\# \bar{\CP}^2$, with maximal square. Then
\[
\scU\cdot F_{B^4\#\bar{\CP}^2,\cS,\frs}\simeq \scU \cdot F_{B^4\# \bar{\CP}^2,\cS,\bar{\frs}}\quad \text{and} \quad \scV\cdot F_{B^4\# \bar{\CP}^2, \cS,\frs}\simeq \scV \cdot F_{B^4\# \bar{\CP}^2, \cS,\bar{\frs}}.
\]
\end{lem}
\begin{proof} Two $\bF[\scU,\scV]$-equivariant maps from $\bF[\scU,\scV]$ to $\cCFL^-(\bH)$ are homotopic if and only if their evaluations on $1$ are homologous. The link Floer complex of the negative Hopf link is shown below:
\begin{equation}
\begin{tikzcd}[labels=description,row sep=1cm, column sep=1cm]
\ve{a}
&
\ve{b}
	\arrow[from=d,"\scV"]
	\arrow[from=l,"\scU"]	
	\\
\ve{c}
	\arrow[from=r, "\scV"]
	\arrow[from=u,"\scU"]
&\ve{d}.
\end{tikzcd}
\label{eq:negative-Hopf}
\end{equation}
Using the grading formulas from \cite{ZemAbsoluteGradings}*{Theorem~1.4}, 
we compute that
\[
\begin{split}
\gr_{\ws}(F_{B^4\# \bar{\CP}^2,\cS,\frs}(1))&=\gr_{\zs}(F_{B^4\# \bar{\CP}^2,\cS,\frs}(1))=\frac{1}{2}, \quad \text{and}\\
\gr_{\ws}(F_{B^4\# \bar{\CP}^2,\cS,\bar{\frs}}(1))&=\gr_{\zs}(F_{B^4\# \bar{\CP}^2,\cS,\bar{\frs}}(1))=\frac{1}{2}.
\end{split}
\]
 The $(\gr_{\ws},\gr_{\zs})$-bigradings of $\ve{a}$, $\ve{b}$, $\ve{c}$ and $\ve{d}$ are  $\left(-\tfrac{1}{2},\tfrac{3}{2}\right)$, $\left(\tfrac{1}{2},\tfrac{1}{2}\right)$, $\left(\tfrac{1}{2},\tfrac{1}{2}\right)$, and $\left(\tfrac{3}{2},-\tfrac{1}{2}\right)$, respectively. Hence $F_{B^4\# \bar{\CP}^2, \cS,\frs}(1)$ and $F_{B^4\# \bar{\CP}^2,\cS,\bar{\frs}}(1)$ are in the $\bF$-span of $\ve{b}$ and $\ve{c}$. Furthermore, neither can be $\ve{b}+\ve{c}$, because $\scU\cdot[\ve{b}+\ve{c}]=0$ and both $F_{B^4\# \bar{\CP}^2, \cS,\frs}(1)$ and $F_{B^4\# \bar{\CP}^2, \cS,\bar{\frs}}(1)$ are non-torsion, by a simple adaptation of \cite{ZemAbsoluteGradings}*{Theorem~1.7}. Hence, each map sends $1$ to either $\ve{b}$ or $\ve{c}$. Since
\[
\scU \cdot [\ve{b}]=\scU\cdot [\ve{c}]\quad \text{and} \quad \scV\cdot [\ve{b}]=\scV\cdot [\ve{c}],
\]
the ambiguity disappears after multiplying by either $\scU$ or $\scV$, and the main claim follows.
\end{proof}

\subsection{Involutive correction terms and the clasp number}

We now prove the clasp number bounds from the involutive correction terms:

\begin{proof}[Proof of equation~\eqref{thm:involutive-clasp} of Theorem~\ref{thm:involutive}]
The proof is similar to the proof of equation~\eqref{thm:involutive-genus}. Suppose that $S$ is an immersed surface in $B^4$, which bounds $K$. Suppose that there are $p$ positive double points, and $n$ negative double points. We blow-up at the negative double points, and smooth the positive double points to obtain a properly embedded, null-homologous surface $\hat{S}$ in $B^4 \#  n\bar{\CP}^2$. We let $\frs\in \Spin^c( B^4 \# n \bar{\CP}^2)$ have maximal square. We asume that $\hat{S}$ has odd genus, by stabilizing if necessary, and then perform an additional stabilization to obtain a surface $S'$, which has even genus. Now, we let $\cF'$ denote a decoration of $S'$ such that the genera of the $\ws$ and $\zs$ subregions are equal, and such that the stabilization torus of $S'$ is contained in the $\ws$ subregion. We consider the map
\[
F_{B^4\# n \bar{\CP}^2, \cF',\frs}\colon \bF[\scU,\scV]\to \cCFK^-(K).
\]
The map induces a $(\gr_{\ws},\gr_{\zs})$-bigrading change of $(-g(S)-1,-g(S)-1)$ by the grading change formulas of Section~\ref{subsec:TQFT}.

We let $\bar{\cF}'$ denote the decoration of $S'$ obtained by reversing the roles of the $\ws$ and $\zs$ subregions, and adding a half twist to the dividing set along the boundary. Similar to~\eqref{eq:involution-commutes}, we have 
\begin{equation}
\iota_K\circ F_{B^4\# n \bar{\CP}^2, \cF',\frs}\eqsim F_{B^4 \# n \bar{\CP}^2, \bar{\cF}',\bar{\frs}}\circ \iota_0.\label{eq:inv-clasp-homotopy-1}
\end{equation}
Proposition~\ref{prop:main-claim-dividing-set-ind} implies that 
\begin{equation}
F_{B^4\# n \bar{\CP}^2,\bar{\cF}',\bar{\frs}}\simeq F_{B^4\# n \bar{\CP}^2,\cF', \bar{\frs}}, \label{eq:inv-clasp-homotopy-2}
\end{equation}
while Lemmas~\ref{lem:stabilization} and~\ref{lem:Hopf-CP2} imply that
\begin{equation}
F_{B^4\# n \bar{\CP}^2,\cF', \bar{\frs}}\simeq F_{B^4\# n\bar{\CP}^2,\cF',\frs}.\label{eq:inv-clasp-homotopy-3}
\end{equation}
Combining equations~\eqref{eq:inv-clasp-homotopy-1}, ~\eqref{eq:inv-clasp-homotopy-2}, and~\eqref{eq:inv-clasp-homotopy-3}, we obtain a skew-equivariant homotopy $H$ between
$\iota_K\circ F_{B^4\# n \bar{\CP}^2, \cF',\frs}$ and $F_{B^4\# n \bar{\CP}^2,\cF',\frs}\circ \iota_0$, which we use to make a map 
\[
F\colon \bF[U,Q]/Q^2\to AI_0(K),
\]
which has homogeneous grading $-p-1$, and which becomes an isomorphism on homology after inverting $U$. It follows that
\[
\underline{V}_0(K)\le \frac{p+1}{2}=\left\lceil \frac{p+1}{2}\right\rceil.
\]

The proof of the stated inequality involving $\overline{V}_0(K)$ and $c^-(K)$ is obtained by turning around and reversing the orientation of $S$, to obtain a normally immersed surface in $B^4$, which we now view as a cobordism from $K$ to $\emptyset$. Note that doing so changes the sign of each double point. We smooth the positive double points (which were originally negative double points), and blow up at the negative double points (which were previously positive). By combining the above argument with the analogous claim in the proof of equation~\eqref{thm:involutive-genus} of Theorem~\ref{thm:involutive}, we obtain the statement.
\end{proof}

\section{An invariant from the $\scU\scV=0$ version of knot Floer homology}

We now describe our invariant $\omega(K)$, which is defined using the version of knot Floer homology obtained by setting $\scU\scV=0$. We write $\widehat{\cR}$ for the ring $\bF[\scU,\scV]/\scU\scV$. In this section, we will write $\widehat{\cC}_K$ for $\cCFK^-(K)/\scU\scV$.

\begin{define}
 We define $\omega (K)\in \N$ to be the minimal $n$ such that there is a $(\gr_{\ws},\gr_{\zs})$-grading preserving, $\widehat{\cR}$-equivariant chain map from $\widehat{\cR}$ to $\widehat{\cC}_K\otimes \widehat{\cS}_n^\vee$, which sends $\bF[\scU]$-non-torsion elements to $\bF[\scU]$-non-torsion elements, and sends $\bF[\scV]$-non-torsion elements to $\bF[\scV]$-non-torsion elements.
 Here $\widehat{\cS}_n^\vee$ is the $\widehat{\cR}$ reduction of the staircase complex used in the definition of $Y_n(K)$.
\end{define} 

We will call a chain map which sends $\bF[\scU]$-non-torsion elements to $\bF[\scU]$-non-torsion elements a $\bF[\scU]$-local map. We define a $\bF[\scV]$-local map similarly.
The most important property satisfied by $\omega(K)$ is the following:

\begin{thm}
 For a knot $K$ in $S^3$,
\[
\omega (K)\le g_4(K).
\]
\end{thm}
\begin{proof} The proof of Theorem~\ref{thm:Y-genus-clasp} implies that if $K$ has a slice surface of genus $n$, then we obtain a grading preserving local chain map
 \[
F\colon \cS_n\to  \cCFK^-(K).
 \]
 Tensoring with the identity map of $\cS_n^\vee$, and composing with the cotrace map from $\cR$ to $\cS_n\otimes \cS_n^\vee$, which is a local map (see \cite{ZemConnectedSums}*{Lemma~2.18}), we obtain a local map from $\cR$ to $\cCFK^-(K)\otimes \cS_n^\vee$. The  reduction to $\widehat{\cR}$ of $F$ is grading preserving and both $\bF[\scU]$-local and $\bF[\scV]$-local.
\end{proof}

We now prove some basic properties of $\omega(K)$, and also compare $\omega(K)$ and $\nu(K)$. We write $\widehat{\scA}_n$ for the subset of $\widehat{\cC}_K$ in Alexander grading $n$. We recall that $\nu (K)$ is defined to be the minimum $s\in \Z$ such that the map
\[
H_*(\widehat{\scA}_s)\to H_* \left(\widehat{\cC}_K/(\scU,\scV-1)\right)\iso \bF
\]
is surjective (see \cite{OSRationalSurgeries}*{Definition~9.1}). We have the following alternate characterizations of $\nu (K)$:
\begin{lem} The integer $\nu (K)$ coincides with the minimal $s$ such that $H_*(\widehat{\scA}_s)$ contains an $\bF[\scV]$-non-torsion element.
\end{lem}
\begin{proof} Firstly, the map $H_*(\widehat{\scA}_s)\to H_*(\widehat{\cC}_K/(\scU,\scV-1))$ factors through the map $H_*(\widehat{\cC}_K/\scU)\to H_*(\widehat{\cC}_K/(\scU,\scV-1))$, which is easily seen to map an element $x$ to $1\in \bF$ if and only if $x$ is $\bF[\scV]$-non-torsion. Hence,  $\nu (K)$ coincides with the minimal $s$ such that
\[ 
\im \left( H_*(\widehat{\scA}_s)\to H_*(\widehat{\cC}_K/\scU) \right)
\]
contains an $\bF[\scV]$-non-torsion element. Finally, it is easy to see that the map $H_*(\widehat{\cC}_K)\to H_*(\widehat{\cC}_K/\scU)$ maps $\bF[\scV]$-non-torsion elements to $\bF[\scV]$-non-torsion elements, which implies the claim.
\end{proof}

We have the following:

\begin{prop} For any knot $K\subset S^3$,
\[
\nu (K)\le \omega (K).
\]
\end{prop}
\begin{proof} Suppose there is a grading preserving chain map $F\colon \widehat{\cR}\to \widehat{\cC}_K\otimes \widehat{\cS}_n^\vee$ which is both $\bF[\scU]$-local and $\bF[\scV]$-local. Consider the chain map from $\widehat{\cS}_n^\vee$ to $\widehat{\cR}$ given by projecting onto $x_{-n}$. Note that this map is $\bF[\scV]$-local, since any $\bF[\scV]$-non-torsion element of $H_*(\widehat{\cS}_n^\vee)$  must contain $\scV^k x_{-n}$ as a summand, for some $k$. The tensor product of two $\bF[\scV]$-local maps is $\bF[\scV]$-local, so we obtain an $\bF[\scV]$-local map from  $\widehat{\cC}_K\otimes \widehat{\cS}_n^\vee$ to $\widehat{\cC}_K$, which increases the Alexander grading by $n$. Composing these maps and evaluating at $1\in \widehat{\cR}$, we obtain an $\bF[\scV]$-non-torsion element of $H_*(\widehat{\scA}_n)$, so $\nu (K)\le n$. The proof is complete.
\end{proof}

\begin{lem}\label{lem:omega-range-values}
 The invariant $\omega (K)$ takes values in $\{\tau(K),\tau(K)+1\}$.
\end{lem}
\begin{proof} 
We can view $\widehat{\cC}_K/\scU$ and $\widehat{\cC}_K/\scV$ as subspaces of $\widehat{\cC}_K$ (though not as subcomplexes). Namely, we view $\widehat{\cC}_K/\scU$ as being generated over $\bF$ by monomials $\scV^j \cdot \xs$ where $j\in \N$ and $\xs\in \bT_{\a}\cap \bT_{\b}$, and similarly for $\widehat{\cC}_K/\scV$.  
We let $y_{\tau}$ be an $\bF[\scV]$-non-torsion cycle of $\widehat{\cC}_K/\scU$ which is of Alexander grading $\tau$, and we let $y_{-\tau}\in \widehat{\cC}_K/\scV$ be an $\bF[\scU]$-non-torsion cycle, of Alexander grading $-\tau$. The elements $y_{\tau}$ and $y_{-\tau}$ may not be cycles in $\widehat{\cC}_K$. However, $\scV\cdot y_{\tau}$ and $\scU\cdot y_{-\tau}$ are cycles in $\widehat{\cC}_K$. We can define a grading preserving, $\widehat{\cR}$-equivariant map
\[
F\colon \cS_{\tau+1}\to \widehat{\cC}_K
\]
 which sends $x_{-\tau-1}$ to $\scU \cdot  y_{-\tau}$, sends $x_{\tau+1}$ to $\scV\cdot y_{\tau}$, and maps all other $x_{i}$ to zero. Noting that
\[
\scU\cdot (\scV\cdot y_{\tau})=0\quad \text{and} \quad \scV\cdot (\scU\cdot y_{-\tau})=0,
\]
we see that the map $F$ is a chain map. The map $F$ is also clearly $\bF[\scU]$-local and $\bF[\scV]$-local. Hence, there is a grading preserving, $\bF[\scU]$ and $\bF[\scV]$-local map from $\widehat{\cS}_{\tau+1}$ to $\widehat{\cC}_K$. Algebraically, this is equivalent to the existence of a grading preserving, $\bF[\scU]$ and $\bF[\scV]$-local map from $\widehat{\cR}$ to  $\widehat{\cC}_K\otimes \widehat{\cS}_{\tau+1}^\vee$, completing the proof.

\end{proof}

\section{Examples}
\label{sec:examples}

In this section we describe a few simple examples. We note that we wrote a computer program using SageMath \cite{SageMath} which can perform all of the stated computations. It is available at the second author's webpage.

\subsection{A knot with $\omega (K)>\nu (K)$}

We now consider the complex $\scC$ which has 3 generators, $\ve{a}$, $\ve{b}$ and $\ve{c}$, with differential 
\[
\d \ve{b} =\scU^2 \cdot \ve{a}+\scV^2\cdot \ve{c}.
\]
The $(\gr_{\ws},\gr_{\zs})$ bigradings of $\ve{a}$, $\ve{b}$, and $\ve{c}$ are $(0,-4)$, $(-3,-3)$, and $(-4,0)$, respectively. It was observed by Hedden and Watson that the complex $\scC$ is not realizable as the knot Floer complex of a knot in $S^3$ \cite{HeddenWatson}*{Theorem~7}. Nonetheless, Hom proved that the complex $\scC$ is locally equivalent to the knot Floer complex for $K=T_{4,5}\# -T_{2,3;2,5}$ (and hence all knot Floer concordance invariants of $K$ may be computed from $\scC$). See \cite{HomEpsilonUpsilon}*{Lemma~2.1}.

\begin{lem}\label{lem:long-trefoil}
 The complex $\scC$ has $\tau(\scC)=\nu (\scC)=2$ and $\omega (\scC)=3$.
\end{lem}
\begin{proof}
We compute $\nu (K)$ first. The complex $\widehat{\scA}_2(K)$ is generated by $\ve{a}$, $\scV^2\cdot\ve{b}$, and $\scV^4\cdot \ve{c}$. The generator $\ve{a}$ is $\bF[\scV]$-non-torsion, so $\nu (K)\le 2$. On the other hand, $\tau(\scC)=2$, so $\nu (\scC)=2$.

Since $\tau(\scC)=2$, we know $\omega (\scC)\in \{2,3\}$ by Lemma~\ref{lem:omega-range-values}. To see it is not 2, we note that $\omega (\scC)=2$ is equivalent to the existence of cycles $z_{-2},z_0,z_2$, of bigrading $(-4,0)$, $(-2,-2)$, and $(0,-4)$, respectively, such that
\[
\scU\cdot [z_{2}]=\scV\cdot [z_0]\quad \text{and} \quad \scU\cdot [z_0]=\scV\cdot [z_{-2}],
\]
such that $z_{2}$ is $\bF[\scV]$-non-torsion, and $z_{-2}$ is $\bF[\scU]$-non-torsion. The homology $H_*(\widehat{\scC})$ is shown in Figure~\ref{fig:24}.
\begin{figure}[H]
	\centering
	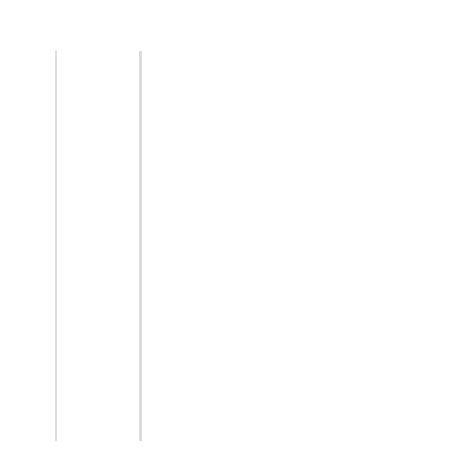
	\caption{The homology group $H_*(\widehat{\scC})$.}\label{fig:24}
\end{figure}

In particular, the module $H_*(\widehat{\scC} )$ is either rank 1 or rank 0 in each bigrading. This forces $z_2=\ve{a}$ and $z_{-2}=\ve{c}$. However $H_*(\widehat{\scC})$ has rank 0 in bigrading $(-2,-2)$, so $z_0=0$. However $\scU\cdot[\ve{a}]\neq 0$ and $\scV\cdot [\ve{c}]\neq 0$, so we obtain a contradiction. Hence $\omega (\scC)=3$.
\end{proof}

\begin{rem}
 The computation of Lemma~\ref{lem:long-trefoil} implies that any knot with complex locally equivalent to $\scC$ has $g_4(K)\ge 3$. We note that $V_0(\scC)=2$, which also implies that $g_4(K)\ge 3$ by~\eqref{eq:rasmussen-bound}. On the other hand, $\omega (K)$ uses only the $\widehat{\cR}$ reduction of $\scC$, and hence also holds for any knot whose $\widehat{\cR}$ reduction is locally equivalent to $\widehat{\scC}$.
\end{rem}

\subsection{A knot where $Y_n(K)$ give better bounds than $V_n(K)$}
\label{sec:computation}
We now consider the knot $K=T_{2,3}\# T_{4,7}\# -T_{5,6}$. The knot complex is the tensor product of the three complexes shown in Figure~\ref{fig:25}.

\begin{figure}[p]
	\centering
	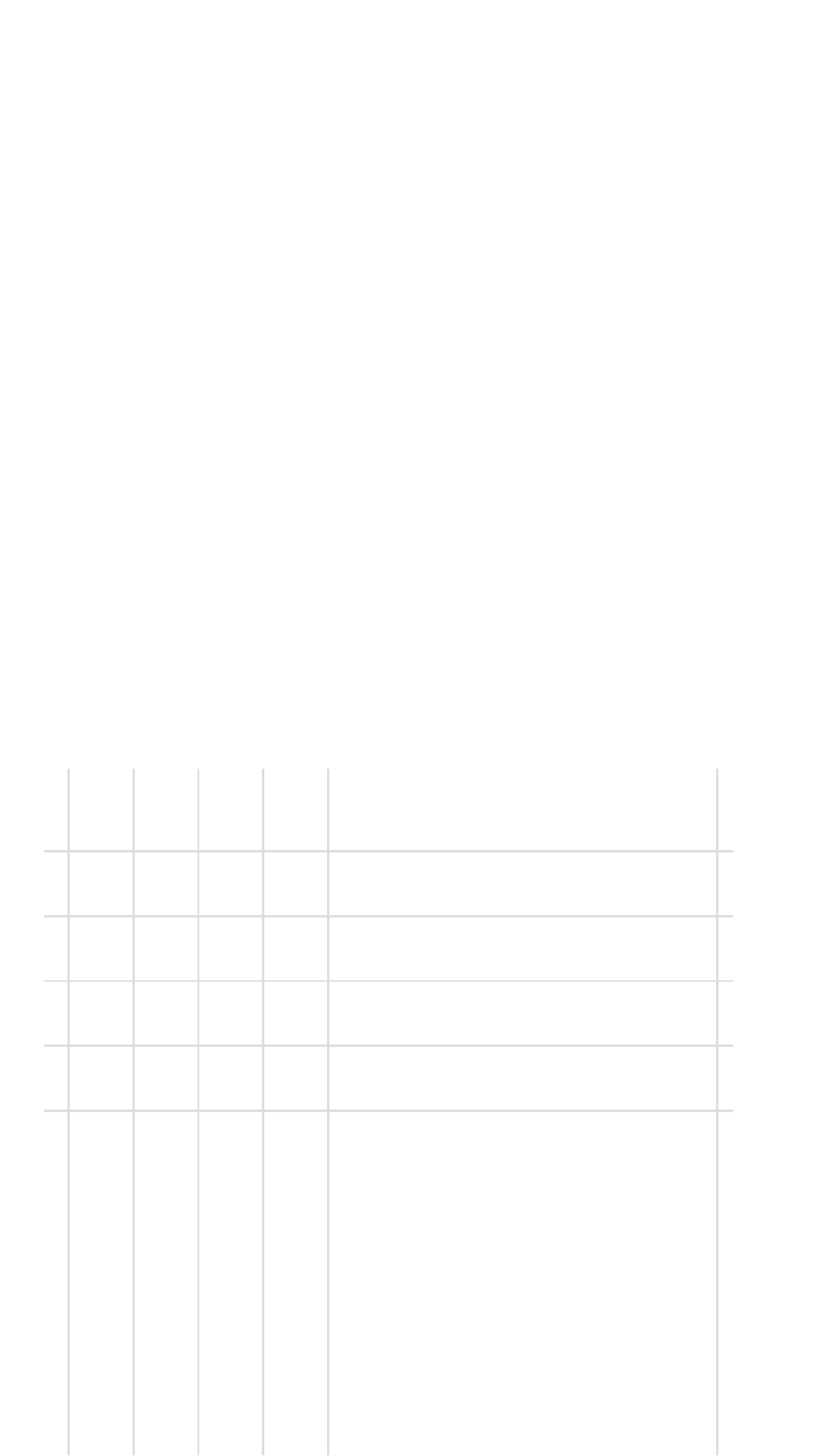
	\caption{The knot Floer complexes of $T_{2,3}$, $T_{4,7}$, and $-T_{5,6}$. The $(\gr_{\ws},\gr_{\zs})$ bigradings of each generator is shown.}\label{fig:25}
\end{figure}

\begin{lem}\label{lem:computation}
 Let $K=T_{2,3}\# T_{4,7}\# -T_{5,6}$.
 \begin{itemize}
 \item  $V_0(K)=1$, and $V_n(K)=0$ for $n\ge 1$.
 \item  $V_0(-K)=1$, and $V_n(-K)=0$ for $n\ge 1$.
 \item $Y_0(K)=1$, $Y_1(K)=1$, and $Y_n(-K)=0$ for $n \ge 1$.
 \end{itemize} In particular, the lower bound on the slice genus from the $V_n$ invariants is $1$, while the lower bound from the $Y_n$ invariants is $2$.
\end{lem}
\begin{proof} We begin by showing that $V_0(K)>0$. For $V_0(K)=0$, we would need an $\bF[U]$-non-torsion cycle in $(\gr_{\ws},\gr_{\zs})$-bigrading $(0,0)$ (we remind the reader of the isomorphism in equation~\eqref{eq:A-isomorphism-with-shifts}). The set of $\bF$ generators of $\cCFK^-(K)$ in grading $(0,0)$ is
\begin{equation}
x_2y_{10}z_0,\quad x_2y_9z_1,\quad x_0y_7z_3,\quad x_2y_3z_5,\quad x_0y_1x_7, \quad \text{and} \quad x_0y_0z_8. \label{eq:elements(0,0)}
\end{equation}
It is easy to see that no linear combination of the elements in~\eqref{eq:elements(0,0)} is an $\bF[U]$-non-torsion cycle. Hence $V_0(K)>0$

Next, we will show that $V_1(K)=0$. Note that this will imply that $V_0(K)=1$, since $V_1(K)\ge V_0(K)-1$. This amounts to finding an $\bF[U]$-non-torsion cycle in bigrading $(0,-2)$. Such an element is
\[
\begin{split}
X=&\scV x_2y_{10}z_0+x_2y_8z_2+x_0y_6z_4+x_2y_2z_6+ \scV x_0y_0z_8\\
 +&\scV x_2 y_9z_1+\scV x_0y_7z_3+\scU x_1y_8z_3\\
+&\scV x_2y_3z_5+\scU\scV x_1y_4z_5+\scU x_0y_5z_5\\
+& \scV x_1 y_0z_7+\scU x_2y_1 z_7.
\end{split}
\]
Indeed, it is easily checked that $X$ is a cycle in grading $(0,-2)$. To see that it is $\bF[U]$-non-torsion, we note that after inverting $U$, each $z_{2i+1}$ becomes null-homologous. Further, the top line of $X$ is homologous to 
\[
\scU^{-10}\scV x_0y_0(\scV^{10} z_0+\scU \scV^{6} z_2+\scU^3 \scV^3 z_4+\scU^{6} \scV z_6+\scU^{10} z_8),
\]
which is clearly the generator of homology after inverting $U$.

Finally, we consider $Y_n(K)$. We will prove $Y_1(K)=1$, and leave the computation that $Y_n(K)=0$ for $n\ge 2$ to the reader. By definition
\[
Y_1(K)=V_0\left(\cS_1^\vee \otimes \cCFK^-(T_{2,3})\otimes \cCFK^-(T_{4,7})\otimes \cCFK^-(-T_{5,6})\right).
\]
Since $\cCFK^-(T_{2,3})\simeq \cS_1$, and $\cS_1^\vee\otimes \cS_1$ is locally equivalent to $\bF[\scU,\scV]$, it is sufficient to show
\[
V_0(T_{4,7}\# -T_{5,6})=1.
\]
To see that $V_0(T_{4,7}\#-T_{5,6})>0$, we note that any $\bF[U]$-non-torsion element must have a summand involving a non-trivial multiple of $y_{j_i}\otimes z_{2i}$, for each $i\in \{1,2,3,4\}$ (i.e., each $z_{2i}$ must be represented in the summand). There is no  $y_n$ such that $\scU^i \scV^j y_n z_4$ has bigrading $(0,0)$ for any $i,j\ge 0$, and hence there is no $\bF[U]$-non-torsion element in bigrading $(0,0)$. Hence $V_0>0$. To see that $V_0=1$, we note that the element
\[
\begin{split}
Z&=\scU \scV^2 y_{10}z_0+\scU\scV y_8z_2+\scU^2 y_6z_4+\scU^3 y_4 z_6+\scU^2\scV y_0z_8\\
&+\scU \scV^2 y_9z_1+\scU^2 \scV y_7z_3+\scU^3 y_5z_5+ \scU^5 y_3z_7+\scU^2 \scV y_1z_7.
\end{split}
\]
is an $\bF[U]$-non-torsion cycle in grading $(-2,-2)$.

We leave the computations of $V_n(-K)$ and $Y_n(-K)$ to the interested reader.
\end{proof}

\begin{rem}\label{rem:more-on-example-knot}
Let $K=T_{2,3}\# T_{4,7}\#-T_{5,6}$, as in Lemma~\ref{lem:computation}.
\begin{enumerate}
\item One may compute that $\underline{V}_0(K)=2$ and $\overline{V}_0(K)=1$, which additionally gives the slice genus bound $g_4\ge 2$, by Theorem~\ref{thm:involutive}. This may be computed using the software on the second author's website.
\item The Levine--Tristram signature function also has a maximum of 2 and minimum of $-4$ (taken over generic $t$), and hence gives the same genus and clasp bounds as the $Y_n$ invariants. An example where the $Y_n$ invariants give an improvement over both the $V_n$ invariants and the Levine--Tristram signature is the knot $J=T_{2,11}\# T_{4,7}\# -T_{5,6}$. Using a SageMath program, the authors computed  that $(V_0,V_1,V_2,V_3,V_4,V_5)$ are $(3,2,2,1,1,0)$ and $\max_{t\in [0,1]} \sigma_J(t)=2$ and $\min_{t\in [0,1]}\sigma_J(t)=-10$ (see \cite{LitherlandTorusKnots} for more on computing the signature function). In particular, these invariants both give the slice bound $g_4\ge 5$. On the other hand,  $(Y_0,Y_1,Y_2,Y_3,Y_4,Y_5,Y_6)$ may be computed to be $(3,2,2,1,1,1,0)$, which gives the slice bound of $g_4\ge 6$.
\end{enumerate}
\end{rem}

\subsection{An algebraic example}

We were only able to find examples of knots where the genus bound from $Y_n$ was at most one greater than the genus bound from the $V_n$ invariants. We are not aware of any algebraic constraint that would imply this. As an example, we consider the knot-like complex in Figure~\ref{fig:26}, which was introduced by Ozsv\'{a}th, Stipsicz, and Szab\'{o} in \cite{OSSUpsilon}*{Figure~6}. (It is not known whether the local equivalence class of $\scC$ represents any knot in $S^3$).  We consider the $V_n$ and $Y_n$ invariants of the complex $\cCFK^-(T_{2,5})\otimes \scC$. For this complex, one may compute that $(V_0,V_1,V_2,V_3)=(2,1,1,0)$ and $(Y_0,Y_1,Y_2,Y_3,Y_4,Y_5)=(2,2,2,1,1,0)$. In this case, the best genus bound from the $V_n$ comes from $V_0$ or $V_2$, which both give $g\ge 3$. The best genus bound from the $Y_n$ comes from $Y_2$ or $Y_4$, which both give $g\ge 5$.

\begin{figure}[H]
	\centering
	%% Creator: Inkscape 1.0 (4035a4fb49, 2020-05-01), www.inkscape.org
%% PDF/EPS/PS + LaTeX output extension by Johan Engelen, 2010
%% Accompanies image file 'fig26.pdf' (pdf, eps, ps)
%%
%% To include the image in your LaTeX document, write
%%   \input{<filename>.pdf_tex}
%%  instead of
%%   \includegraphics{<filename>.pdf}
%% To scale the image, write
%%   \def\svgwidth{<desired width>}
%%   \input{<filename>.pdf_tex}
%%  instead of
%%   \includegraphics[width=<desired width>]{<filename>.pdf}
%%
%% Images with a different path to the parent latex file can
%% be accessed with the `import' package (which may need to be
%% installed) using
%%   \usepackage{import}
%% in the preamble, and then including the image with
%%   \import{<path to file>}{<filename>.pdf_tex}
%% Alternatively, one can specify
%%   \graphicspath{{<path to file>/}}
%% 
%% For more information, please see info/svg-inkscape on CTAN:
%%   http://tug.ctan.org/tex-archive/info/svg-inkscape
%%
\begingroup%
  \makeatletter%
  \providecommand\color[2][]{%
    \errmessage{(Inkscape) Color is used for the text in Inkscape, but the package 'color.sty' is not loaded}%
    \renewcommand\color[2][]{}%
  }%
  \providecommand\transparent[1]{%
    \errmessage{(Inkscape) Transparency is used (non-zero) for the text in Inkscape, but the package 'transparent.sty' is not loaded}%
    \renewcommand\transparent[1]{}%
  }%
  \providecommand\rotatebox[2]{#2}%
  \newcommand*\fsize{\dimexpr\f@size pt\relax}%
  \newcommand*\lineheight[1]{\fontsize{\fsize}{#1\fsize}\selectfont}%
  \ifx\svgwidth\undefined%
    \setlength{\unitlength}{97.95218694bp}%
    \ifx\svgscale\undefined%
      \relax%
    \else%
      \setlength{\unitlength}{\unitlength * \real{\svgscale}}%
    \fi%
  \else%
    \setlength{\unitlength}{\svgwidth}%
  \fi%
  \global\let\svgwidth\undefined%
  \global\let\svgscale\undefined%
  \makeatother%
  \begin{picture}(1,0.9337716)%
    \lineheight{1}%
    \setlength\tabcolsep{0pt}%
    \put(0,0){\includegraphics[width=\unitlength,page=1]{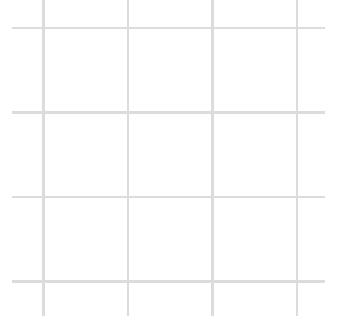}}%
    \put(0.89005917,0.45634231){\color[rgb]{0,0,0}\makebox(0,0)[lt]{\lineheight{1.25}\smash{\begin{tabular}[t]{l}$\scriptstyle{\scV^3}$\end{tabular}}}}%
    \put(0.49759433,0.87737349){\color[rgb]{0,0,0}\makebox(0,0)[t]{\lineheight{1.25}\smash{\begin{tabular}[t]{c}$\scriptstyle{\scU^3}$\end{tabular}}}}%
    \put(0,0){\includegraphics[width=\unitlength,page=2]{fig26.pdf}}%
    \put(0.49759433,0.0133321){\color[rgb]{0,0,0}\makebox(0,0)[t]{\lineheight{1.25}\smash{\begin{tabular}[t]{c}$\scriptstyle{\scU^3}$\end{tabular}}}}%
    \put(0.10989037,0.45634231){\color[rgb]{0,0,0}\makebox(0,0)[rt]{\lineheight{1.25}\smash{\begin{tabular}[t]{r}$\scriptstyle{\scV^3}$\end{tabular}}}}%
    \put(0.70715819,0.70572528){\color[rgb]{0,0,0}\makebox(0,0)[rt]{\lineheight{1.25}\smash{\begin{tabular}[t]{r}$\scriptstyle{\scU\scV}$\end{tabular}}}}%
  \end{picture}%
\endgroup%

	\caption{Ozsv\'{a}th, Stipsicz and Szab\'{o}'s knot-like complex $\scC$.}\label{fig:26}
\end{figure}

\subsection{Further questions}

Here are a few interesting questions:
\begin{enumerate}
\item Are there knots such that $c^+_4(K)-g_4(K)$ is arbitrarily large? \footnote{Daemi and Scaduto's recent work \cite{DaemiScadutoClasp} implies that this is true for the family $\#^n 7_4$.}
\item The figure-8 knot can be unknotted with either a positive or a negative crossing change. Hence $c_4>0$ while $c_4^+=c_4^-=0$. In general, can the difference between $c_4(K)$ and $c_4^+(K)+c_4^-(K)$ be arbitrarily large?
\item Are there knots where $\omega^+(K)>\nu^+(K)+1$.
%\item Is there an invariant, defined using Khovanov homology, which also gives a lower bound on the 4-dimensional clasp number?
\item Are there any knot Floer concordance invariants which can be used to show that a knot has $c^+(K)>g_4(K)$?
\end{enumerate}

\bibliographystyle{plain}
\bibliography{biblio}

% \bib, bibdiv, biblist are defined by the amsrefs package.
\begin{bibdiv}
\begin{biblist}

\bib{AE}{article}{
      author={Alishahi, Akram},
      author={Eftekhary, Eaman},
       title={Tangle floer homology and cobordisms between tangles},
        date={2016},
        note={e-print, \url{arXiv:1610.07122}},
}

\bib{BCGCobordisms}{article}{
      author={Bodn\'{a}r, J\'{o}zsef},
      author={Celoria, Daniele},
      author={Golla, Marco},
       title={A note on cobordisms of algebraic knots},
        date={2017},
        ISSN={1472-2747},
     journal={Algebr. Geom. Topol.},
      volume={17},
      number={4},
       pages={2543\ndash 2564},
}

\bib{DaemiScadutoClasp}{article}{
      author={Daemi, Aliakbar},
      author={Scaduto, Christopher},
       title={Chern-{S}imons functional, singular instantons, and the
  four-dimensional clasp number},
        date={2020},
        note={e-print, \url{arXiv:2007.13160}},
}

\bib{FellerParkClasp}{unpublished}{
      author={Feller, Peter},
      author={Park, JungHwan},
       title={A note on the four-dimensional clasp number of knots},
        date={2020},
        note={e-print, \url{arXiv:2009.01815}},
}

\bib{HKMTQFT}{unpublished}{
      author={Honda, Ko},
      author={Kazez, William},
      author={Mati\'c, Gordana},
       title={Contact structures, sutured {F}loer homology and {TQFT}},
        date={2008},
        note={e-print, \url{arXiv:0807.2431}},
}

\bib{HMInvolutive}{article}{
      author={Hendricks, Kristen},
      author={Manolescu, Ciprian},
       title={Involutive {H}eegaard {F}loer homology},
        date={2017},
     journal={Duke Math. J.},
      volume={166},
      number={7},
       pages={1211\ndash 1299},
}

\bib{HomEpsilonUpsilon}{article}{
      author={Hom, Jennifer},
       title={A note on the concordance invariants epsilon and upsilon},
        date={2016},
        ISSN={0002-9939},
     journal={Proc. Amer. Math. Soc.},
      volume={144},
      number={2},
       pages={897\ndash 902},
         url={https://doi.org/10.1090/proc/12706},
      review={\MR{3430863}},
}

\bib{HomSurvey}{article}{
      author={Hom, Jennifer},
       title={A survey on {H}eegaard {F}loer homology and concordance},
        date={2017},
        ISSN={0218-2165},
     journal={J. Knot Theory Ramifications},
      volume={26},
      number={2},
       pages={1740015, 24},
         url={https://doi.org/10.1142/S0218216517400156},
      review={\MR{3604497}},
}

\bib{HomWuNu+}{article}{
      author={Hom, Jennifer},
      author={Wu, Zhongtao},
       title={Four-ball genus and a refinement of the {O}zsv\'ath-{S}zab\'o tau
  invariant},
        date={2016},
     journal={J. Symplectic Geom.},
      volume={14},
      number={1},
       pages={305\ndash 323},
}

\bib{HeddenWatson}{article}{
      author={Hedden, Matthew},
      author={Watson, Liam},
       title={On the geography and botany of knot {F}loer homology},
        date={2018},
        ISSN={1022-1824},
     journal={Selecta Math. (N.S.)},
      volume={24},
      number={2},
       pages={997\ndash 1037},
         url={https://doi.org/10.1007/s00029-017-0351-5},
      review={\MR{3782416}},
}

\bib{JCob}{article}{
      author={Juh\'asz, Andr\'as},
       title={Cobordisms of sutured manifolds and the functoriality of link
  {F}loer homology},
        date={2016},
        ISSN={0001-8708},
     journal={Adv. Math.},
      volume={299},
       pages={940\ndash 1038},
}

\bib{JZStabilization}{unpublished}{
      author={Juh\'asz, Andr\'as},
      author={Zemke, Ian},
       title={Stabilization distance bounds from link {F}loer homology},
        date={2018},
        note={e-print, \url{arXiv:1810.09158}},
}

\bib{JuhaszZemkeContactHandles}{article}{
      author={Juh\'{a}sz, Andr\'{a}s},
      author={Zemke, Ian},
       title={Contact handles, duality, and sutured {F}loer homology},
        date={2020},
        ISSN={1465-3060},
     journal={Geom. Topol.},
      volume={24},
      number={1},
       pages={179\ndash 307},
}

\bib{KH-Instantons-concordance}{article}{
      author={Kronheimer, Peter},
      author={Mrowka, Tomasz},
       title={Instantons and some concordance invariants of knots},
        date={2019},
        note={e-print, \url{arXiv:1910.11129}},
}

\bib{KimParkInfiniteRank}{article}{
      author={Kim, Min~Hoon},
      author={Park, Kyungbae},
       title={An infinite-rank summand of knots with trivial {A}lexander
  polynomial},
        date={2018},
        ISSN={1527-5256},
     journal={J. Symplectic Geom.},
      volume={16},
      number={6},
       pages={1749\ndash 1771},
         url={https://doi.org/10.4310/JSG.2018.v16.n6.a5},
      review={\MR{3934241}},
}

\bib{LitherlandTorusKnots}{inproceedings}{
      author={Litherland, R.~A.},
       title={Signatures of iterated torus knots},
        date={1979},
   booktitle={Topology of low-dimensional manifolds ({P}roc. {S}econd {S}ussex
  {C}onf., {C}helwood {G}ate, 1977)},
      series={Lecture Notes in Math.},
      volume={722},
   publisher={Springer, Berlin},
       pages={71\ndash 84},
      review={\MR{547456}},
}

\bib{LivingstonTL}{article}{
      author={Livingston, Charles},
       title={Signature invariants related to the unknotting number},
        date={2020},
     journal={Pacific J. Math.},
      volume={305},
      number={1},
       pages={229\ndash 250},
}

\bib{LivingstonVanCott}{article}{
      author={Livingston, Charles},
      author={Van~Cott, Cornelia~A.},
       title={The four-genus of connected sums of torus knots},
        date={2018},
        ISSN={0305-0041},
     journal={Math. Proc. Cambridge Philos. Soc.},
      volume={164},
      number={3},
       pages={531\ndash 550},
}

\bib{MOIntegerSurgery}{unpublished}{
      author={Manolescu, Ciprian},
      author={Ozsv{\'a}th, Peter~S.},
       title={Heegaard {F}loer homology and surgeries on links},
        date={2010},
        note={e-print, \url{arXiv:1011.1317}},
}

\bib{MY}{article}{
      author={Murakami, Hitoshi},
      author={Yasuhara, Akira},
       title={Four-genus and four-dimensional clasp number of a knot},
        date={2000},
     journal={Proc. Amer. Math. Soc.},
      volume={128},
       pages={3693\ndash 3699},
}

\bib{OSIntersectionForms}{article}{
      author={Ozsv{\'a}th, Peter~S.},
      author={Szab{\'o}, Zolt{\'a}n},
       title={Absolutely graded {F}loer homologies and intersection forms for
  four-manifolds with boundary},
        date={2003},
     journal={Adv. Math.},
      volume={173},
      number={2},
       pages={179\ndash 261},
}

\bib{OSDisks}{article}{
      author={Ozsv\'ath, Peter},
      author={Szab\'o, Zolt\'an},
       title={Holomorphic disks and topological invariants for closed
  three-manifolds},
        date={2004},
     journal={Ann. of Math. (2)},
      volume={159},
      number={3},
       pages={1027\ndash 1158},
}

\bib{OSTriangles}{article}{
      author={Ozsv\'ath, Peter},
      author={Szab\'o, Zolt\'an},
       title={Holomorphic triangles and invariants for smooth four-manifolds},
        date={2006},
     journal={Adv. Math.},
      volume={202},
      number={2},
       pages={326\ndash 400},
}

\bib{OSLinks}{article}{
      author={Ozsv\'ath, Peter},
      author={Szab\'o, Zolt\'an},
       title={Holomorphic disks, link invariants and the multi-variable
  {A}lexander polynomial},
        date={2008},
     journal={Algebr. Geom. Topol.},
      volume={8},
      number={2},
       pages={615\ndash 692},
}

\bib{OSRationalSurgeries}{article}{
      author={Ozsv\'ath, Peter},
      author={Szab\'o, Zolt\'an},
       title={Knot {F}loer homology and rational surgeries},
        date={2011},
     journal={Algebr. Geom. Topol.},
      volume={11},
      number={1},
       pages={1\ndash 68},
}

\bib{OwensStrle}{article}{
      author={Owens, Brendan},
      author={Strle, Sa\v{s}o},
       title={Immersed disks, slicing numbers and concordance unknotting
  numbers},
        date={2016},
        ISSN={1019-8385},
     journal={Comm. Anal. Geom.},
      volume={24},
      number={5},
       pages={1107\ndash 1138},
}

\bib{OSBorderedKauffmanStates}{article}{
      author={Ozsv\'{a}th, Peter},
      author={Szab\'{o}, Zolt\'{a}n},
       title={Kauffman states, bordered algebras, and a bigraded knot
  invariant},
        date={2018},
        ISSN={0001-8708},
     journal={Adv. Math.},
      volume={328},
       pages={1088\ndash 1198},
         url={https://doi.org/10.1016/j.aim.2018.02.017},
      review={\MR{3771149}},
}

\bib{SzaboProgram}{misc}{
      author={Ozsv\'ath, Peter},
      author={Szab\'o, Zolt\'an},
       title={Knot {F}loer homology calculator},
         how={\url{https://web.math.princeton.edu/~szabo/HFKcalc.html}},
        note={\url{https://web.math.princeton.edu/~szabo/HFKcalc.html}.
  Accessed: 2019-03-30},
}

\bib{OSSUpsilon}{article}{
      author={Ozsv\'ath, Peter~S.},
      author={Stipsicz, Andr\'as~I.},
      author={Szab\'o, Zolt\'an},
       title={Concordance homomorphisms from knot {F}loer homology},
        date={2017},
        ISSN={0001-8708},
     journal={Adv. Math.},
      volume={315},
       pages={366\ndash 426},
}

\bib{RasmussenKnots}{thesis}{
      author={Rasmussen, Jacob},
       title={{F}loer homology and knot complements},
        type={Ph.D. Thesis},
        date={2003},
        note={\url{arXiv:math/0306378}},
}

\bib{RasmussenGodaTeragaito}{article}{
      author={Rasmussen, Jacob},
       title={Lens space surgeries and a conjecture of {G}oda and {T}eragaito},
        date={2004},
     journal={Geom. Topol.},
      volume={8},
       pages={1013\ndash 1031},
}

\bib{SarkarMovingBasepoints}{article}{
      author={Sarkar, Sucharit},
       title={Moving basepoints and the induced automorphisms of link {F}loer
  homology},
        date={2015},
     journal={Algebr. Geom. Topol.},
      volume={15},
      number={5},
       pages={2479\ndash 2515},
}

\bib{SageMath}{manual}{
      author={Sage~Developers, The},
       title={{S}agemath, the {S}age {M}athematics {S}oftware {S}ystem
  ({V}ersion 8.2)},
        date={2018},
        note={{\tt http://www.sagemath.org}},
}

\bib{Shibuya}{article}{
      author={Shibuya, Tetsuo},
       title={Some relations among various numerical invariants for links},
        date={1974},
     journal={Osaka J. Math.},
      volume={11},
       pages={313\ndash 322},
}

\bib{ZemQuasi}{article}{
      author={Zemke, Ian},
       title={Quasistabilization and basepoint moving maps in link {F}loer
  homology},
        date={2017},
        ISSN={1472-2747},
     journal={Algebr. Geom. Topol.},
      volume={17},
      number={6},
       pages={3461\ndash 3518},
}

\bib{ZemConnectedSums}{article}{
      author={Zemke, Ian},
       title={Connected sums and involutive knot {F}loer homology},
        date={2019},
        ISSN={0024-6115},
     journal={Proc. Lond. Math. Soc. (3)},
      volume={119},
      number={1},
       pages={214\ndash 265},
}

\bib{ZemAbsoluteGradings}{article}{
      author={Zemke, Ian},
       title={Link cobordisms and absolute gradings on link {F}loer homology},
        date={2019},
        ISSN={1663-487X},
     journal={Quantum Topol.},
      volume={10},
      number={2},
       pages={207\ndash 323},
}

\bib{ZemCFLTQFT}{article}{
      author={Zemke, Ian},
       title={Link cobordisms and functoriality in link {F}loer homology},
        date={2019},
        ISSN={1753-8416},
     journal={J. Topol.},
      volume={12},
      number={1},
       pages={94\ndash 220},
}

\end{biblist}
\end{bibdiv}
\end{document}